\documentclass[12pt]{amsart}
\usepackage{amsthm}
\usepackage{amssymb}
\usepackage{wasysym} 
\usepackage{pifont}
\usepackage{graphicx}
\usepackage{psfrag}
\usepackage{datetime}
\usepackage{rotating}

\usepackage[usenames, dvipsnames]{xcolor}

\usepackage{subfigure}
\usepackage{geometry}
\usepackage[all]{xy}

\usepackage{hyperref}

\let\emptyset\varnothing

\title[Scott map, tilings, and flip equivalence classes]{
The fibres of the Scott map on polygon tilings 
are the flip equivalence classes}
\author{Karin Baur$^*$}
\address{${}^*$ Department of Mathematics and scientific computing, University of Graz, Nawi Graz, 8010 Graz, Austria}
\author{Paul P. Martin$^{**}$}
\address{${}^{**}$ Department of Pure Mathematics, University of Leeds, Leeds LS2 9JT, UK}

\begin{document}



\newtheorem{lm}{Lemma}[section]
\newtheorem*{lemma}{Lemma}
\newtheorem{prop}[lm]{Proposition}

\newtheorem*{corollary}{Corollary}
\newtheorem{cor}[lm]{Corollary}
\newtheorem{theorem}[lm]{Theorem}
\newtheorem*{thm}{Theorem}
\newtheorem{conj}[lm]{Conjecture}

\theoremstyle{definition}
\newtheorem{defn}[lm]{Definition}
\newtheorem*{defini}{Definition}
\newtheorem*{notation}{Notation}
\newtheorem{claim}[lm]{Claim}
\newtheorem{ex}[lm]{Example}
\newtheorem*{exas}{Beispiele}
\newtheorem*{rem}{Remark}
\newtheorem{remark}[lm]{Remark}
\newtheorem{qu}{Question}
\newtheorem*{obs}{Observations}
\newtheorem{para}[lm]{}
\newtheorem{idea}[lm]{Idea}

\newcommand{\mdef}{\refstepcounter{lm}\medskip\noindent{\bf (\thelm) } }


\renewcommand{\Re}{{\mathbb R}}  

\newcommand{\RR}{\operatorname{\mathbb{R}}\nolimits}
\newcommand{\Z}{\operatorname{\mathbb{Z}}\nolimits}

\newcommand{\ttt}{\tau}  
\newcommand{\Anii}{A_{n,(i,i+1)}}
\newcommand{\Scott}{{\boldsymbol\sigma}}  
\newcommand{\Scottt}{{\protect\overrightarrow{\boldsymbol\sigma}}\!}  

\def\commentK#1{\textcolor{blue}{ [\texttt{#1}]}}  

\newcommand{\qqq}[1]{\footnote{\textcolor{blue}{#1}}}
\newcommand{\ignore}[1]{}
\newcommand{\rank}{order}
\newcommand{\triangular}{flip}
\newcommand{\triangulation}{flip}
\newcommand{\tiling}{strand}
\newcommand{\regular}{regular}
\newcommand{\ul}[1]{\underline{#1}}
\newcommand{\Power}{{\mathcal P}}
\newcommand{\gam}{\gamma}
\newcommand{\mdefn}[1]{\begin{defn} #1 \end{defn}} 

\newcommand{\mat}[1]{\left( \begin{array}{#1} }
\newcommand{\tam}{\end{array} \right) }

\newcommand\scalemath[2]{\scalebox{#1}{\mbox{\ensuremath{\displaystyle #2}}}}


\newcommand{\SM}{(S,M)} 
\newcommand{\ASM}{A\SM} 
\newcommand{\del}{\Delta} 
\newcommand{\FST}{\cite{FominShapiroThurston}}

\newcommand{\N}{{\mathbb N}}
\newcommand{\pp}{p} 
\newcommand{\profile}{profile}
\newcommand{\oprofile}{ordered profile}
\newcommand{\Md}{M_\partial}  

\newcommand{\SMP}{(S,\underline{m},\pp)}

\newcommand{\te}[1]{[ #1 ]_{\Delta}}  
\newcommand{\ASD}{alternating strand diagram}

\newcommand{\Sym}{\Sigma}  

\newcommand{\srep}{stellar-replacement}
\newcommand{\Gst}{{\mathsf G}}

\newcommand{\pgr}{\gamma}
\newcommand{\Pb}{{\mathsf {Pl}}} 
\newcommand{\Pbr}{{\mathsf {Pb}}} 
\newcommand{\Pf}{{\mathsf {Pf}}}  
\newcommand{\Pfim}{{\mathsf {Pt}}}  
\newcommand{\bouquet}{bouquet} 
\newcommand{\Gprime}{\Gst'}
\newcommand{\mm}{{\mathsf m}}  

\newcommand{\ppm}[1]{\textcolor{green}{#1}}
\newcommand{\kb}[1]{\textcolor{orange}{#1}}
\newcommand{\kub}[1]{\textcolor{blue}{#1}}
\newcommand{\face}{face}  

\newcommand{\fd}{./figd/}  
\newcommand{\figgdir}{./fig/} 
\newcommand{\figg}[2]{\includegraphics[width=#1cm]{\figgdir #2}}
\newcommand{\includegraphicss}[2][]{\includegraphics[#1]{\fd #2}}

\newcommand{\pfig}[2]{\begin{figure}
    \includegraphics[width=5cm]{./fignew/#1.eps}
    \caption{#2 \label{fig:#1}}
  \end{figure}
}





\newcommand{\newfigd}{fignew/}  


\maketitle

\begin{abstract}
We define a map from 
tilings of surfaces with marked points
to strand diagrams,
generalising Scott's construction for the case of  
triangulations of polygons. 
We thus obtain a map from tilings of surfaces to permutations of 
the marked points on boundary components, the {\em Scott map}. 
In the disk case  (polygon tilings) we prove that the 
fibres of the Scott map are the flip equivalence classes. 

The result allows us to consider 
the size of the image as a generalisation of a classical combinatorial
problem.
We hence  determine the size 
in low ranks. 
\end{abstract}



\section{Introduction}\label{ss:intro}



In a groundbreaking paper \cite{scott} Scott 
proves that the homogeneous coordinate ring of a Grassmannian
has a cluster algebra structure. In the process Scott gives 
a construction for Postnikov diagrams \cite{Postnikov} 
starting from triangular tilings of polygons. 
Given a triangulation $T$, one decorates each triangle with `strands'
\[
\figg{1.25}{tri1} \leadsto\, \figg{1.25}{tri1s}
\hspace{1in}
\figg{1.7}{tri2} \leadsto\ \figg{1.7}{tri2s}
\]
The resultant strand diagram $\Scottt(T)$ 
varies depending on the tiling, but 
induces a permutation $\Scott(T)$ on the polygon vertex set 
that is the same permutation in each case. 
This construction is amenable to generalisation
in a number of ways. 
For example, starting with the notion of triangulation for an arbitrary
marked surface $S$ 
\cite{FockGoncharov,FominShapiroThurston}
(the polygon case extended to include handles, multiple boundary
components,
and interior vertices)
there is a simplicial complex $A(S)$ of tilings 
\cite{Harer85,Hatcher,Ivanov} 
of which the
triangulations are the top dimensional simplices. 
Lower simplices/tilings are obtained by deleting edges from a triangulation.
There is a strand diagram $\Scottt(T)$ in each case
(we define it below, see Figure~\ref{fig:first1}%
(a,b)
for the heuristic).
Thus each marked surface induces a subset $\Scott(A(S))$ of the set of
permutations of its boundary vertices
(see Figures~\ref{fig:hexagon}, \ref{fig:p14} for examples). 


This construction gives rise to a number of questions. The one we
address here is, what are the fibres of this Scott map $\Scott$.
To give an intrinsic characterisation 
is a difficult  
problem in general.
Here we give the answer in
the polygon case, i.e. generalising $\Scott(T)$,
with fibre the set of all triangulations of the polygon, 
to the full $A$-complex of the polygon. 

The answer is in terms of another crucial geometrical device used in 
the theory of cluster mutations \cite{FominShapiroThurston} 
and widely elsewhere (see e.g. 
\cite{Hatcher,FockGoncharov, 
Kirillov,BalsamKirillov10}
and cf. \cite{Buchstaber15})
--- flip equivalence (or the {\em Whitehead move}):
\[
\figg{1.7}{tri2} \leadsto\, \figg{1.7}{tri3}
\] 
Our main Theorem, Theorem~\ref{th:main}, can now be stated 
informally
as in the title. 


We shall conclude this 
introduction 
with some further remarks about related work.
Then from \S\ref{sec:defq} - \S\ref{sec:subdiv-to-permut}
we turn to the precise definitions, 
 formal statement and proof
of Thm.\ref{th:main}.

In \S\ref{sec:count} we report on combinatorial aspects of the problem
--- specifically the size of the  
image of the map $\Scott$
in the polygon cases.
The number of triangulations of polygons is given by the Catalan numbers.
Taking the set $A_r$ of all 
tilings
of the $r$-gon, we have
the little Schr\"oder numbers (see e.g. \cite[Ch.6]{Stanley97})
The image-side problem is open. We use
solutions to Schr\"oder's problem and related problems posed by
Cayley (as in \cite{ps,ppr}), and 
our Theorem to compute the
sequence in low rank $r$, and 
in \S\ref{sec:AEn} prove a key Lemma towards the general
problem. 
To give a flavour of the set $\Scott(A_r) \subset \Sym_r$, the set of vertex permutations:
\[
| \Scott(A_r) | =  1,2,7,26,100,404,1691, \dots \qquad (r=3,4,5,\dots,9)
\]
%


Finally in \S\ref{ss:new-8} 
we give some elementary applications of 
Theorem~\ref{th:main} to Postnikov's alternating strand diagrams and the closely
related reduced plabic graphs \cite{Postnikov}. 
In particular we consider a direct map $\Gst$ from
tilings to plabic graphs 
generalising \cite[\S2]{Propp05}.  
(A heuristic for this `\srep' map is given by the examples
{ \newcommand{\lenf}{1.8cm}
\[
\includegraphics[width=\lenf]{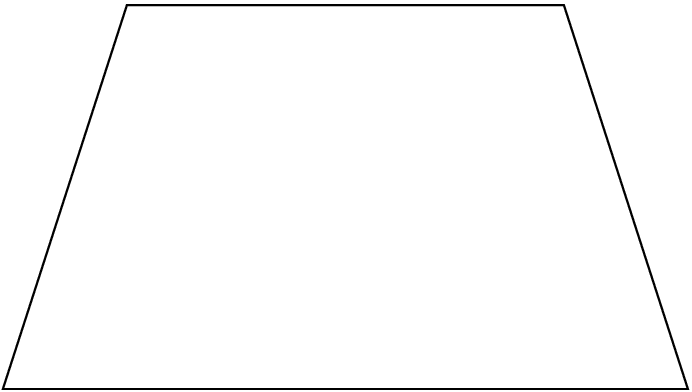}
\leadsto
\includegraphics[width=\lenf]{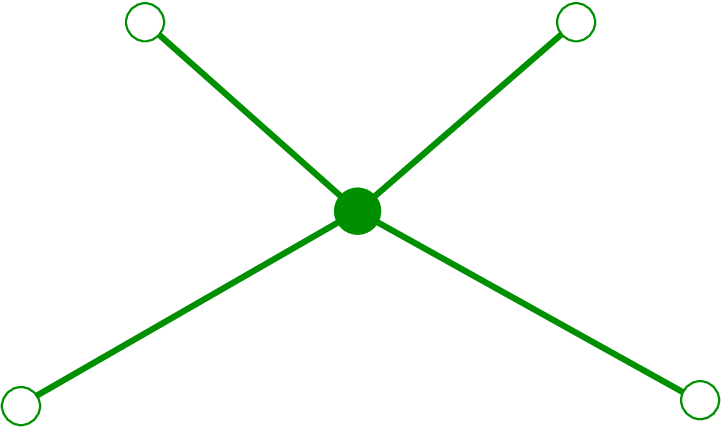}
\hspace{2.1cm}  
\includegraphics[width=2.6cm]{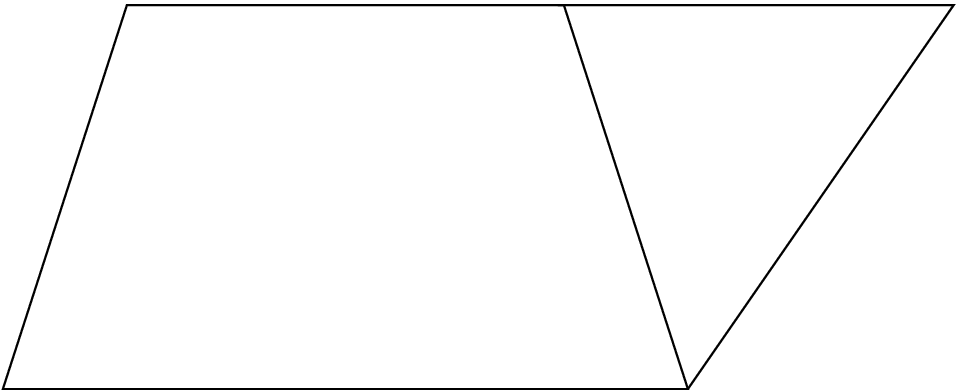}
\leadsto
\includegraphics[width=2.6cm]{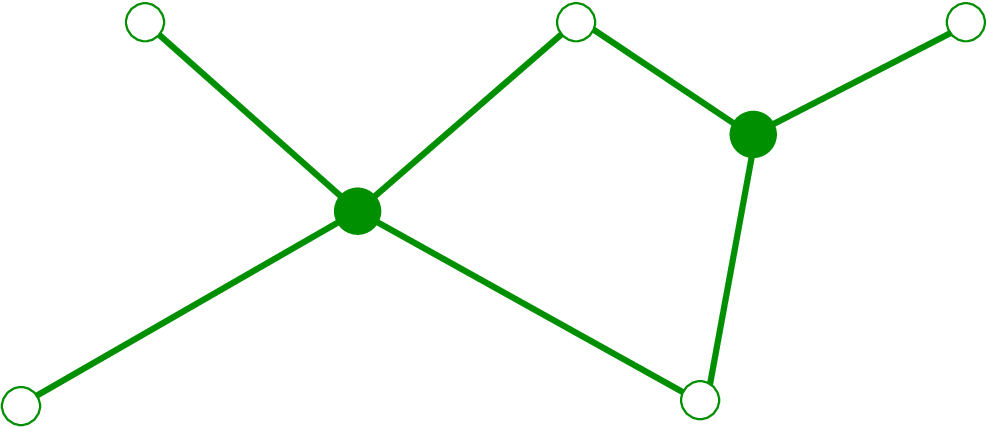}
\]
}
and then Figure~\ref{fig:first1}(c).) 



\begin{figure}
\includegraphics[scale=.3]{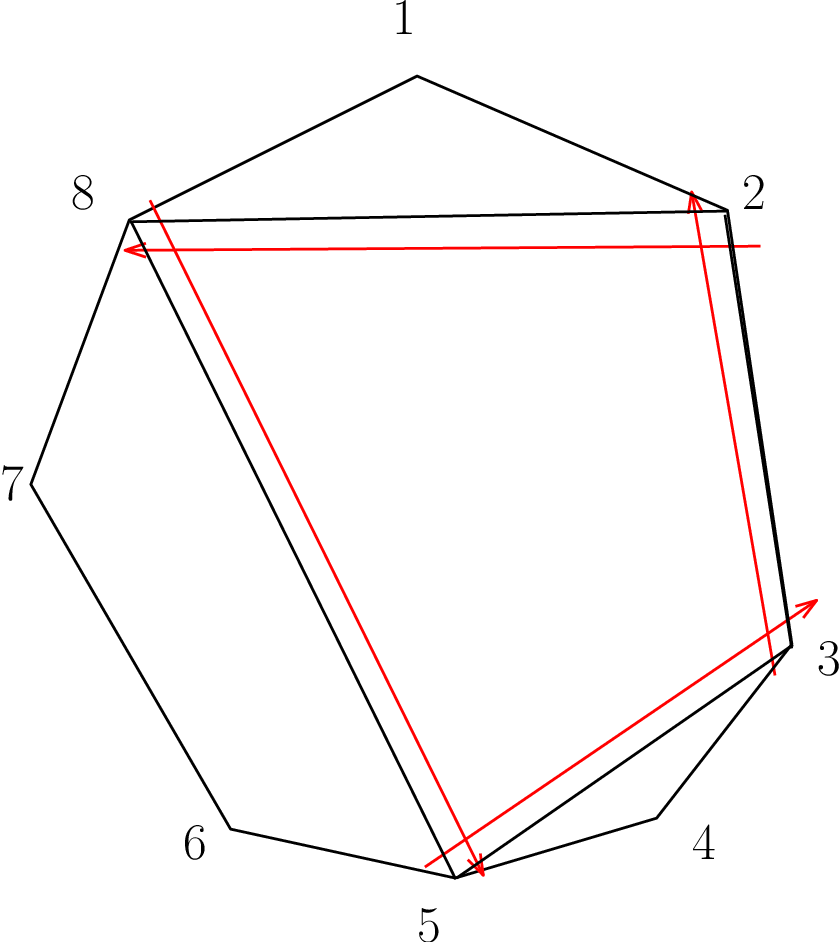}
\hskip .51cm
\includegraphics[scale=.3]{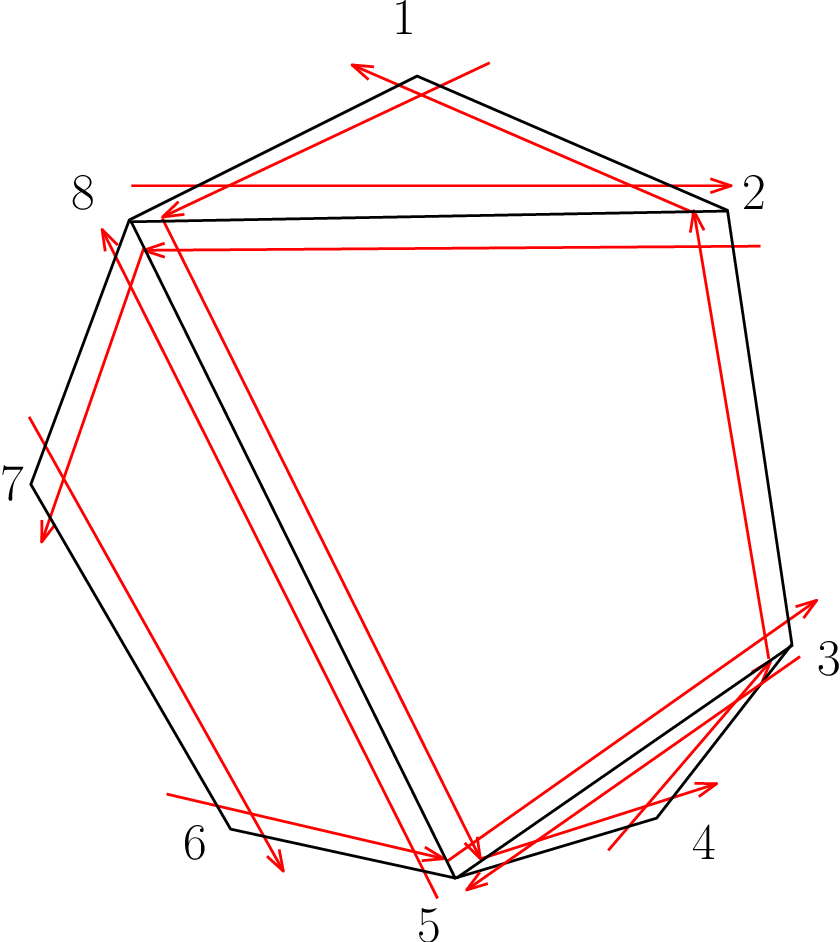}
\hskip .51cm
\includegraphics[scale=.3]{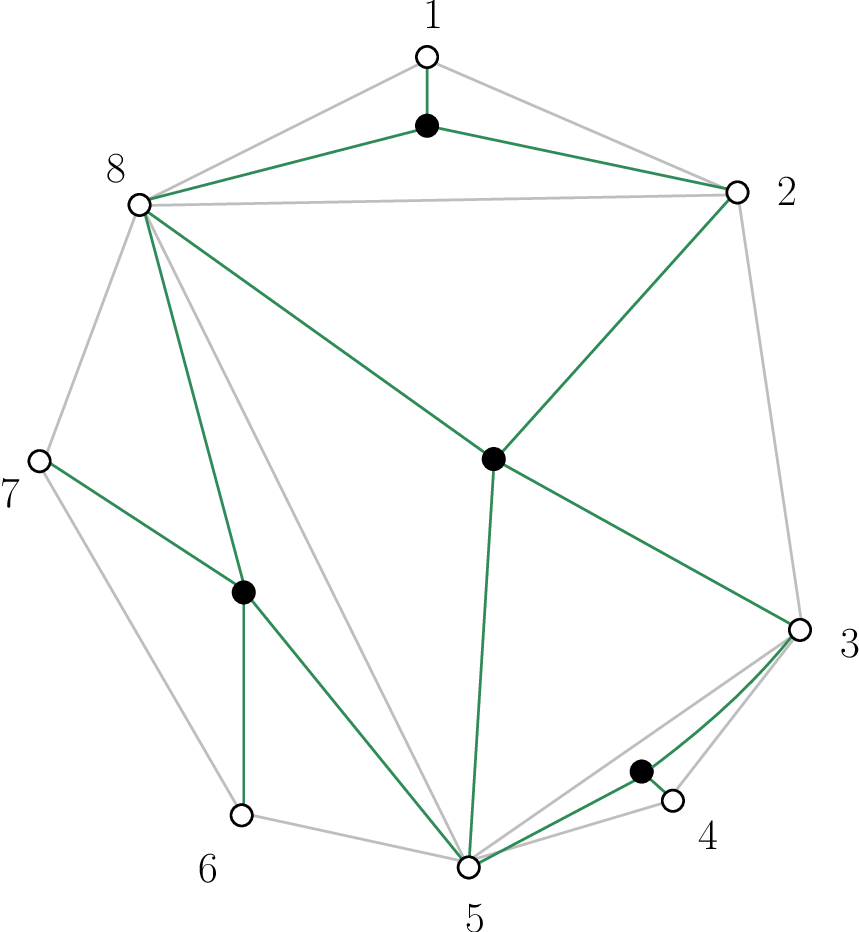}
\caption{(a) Tile with strand segments; 
(b) tiling with strands;
(c) induced plabic graph.
\label{fig:first1}}
\end{figure}



The geometry and topology of the plane, and of two-manifolds,
continues to reward study from several perspectives. 
Recent motivations include 
modelling of anyons for Topological Quantum Computation \cite{Kitaev}, 
fusion categories \cite{BalsamKirillov10},
cluster categories \cite{BaurKingMarsh, BruestleZhang}, 
Teichm\"uller spaces \cite{FockGoncharov,Ivanov},
frieze patterns \cite{Bessenrodt},
diagram algebras \cite{TL321perm,KadarMartinYu,MarshMartin}, 
classical
problems in combinatorics \cite{LandoZvonkin04,ps} and 
combinatorics of symmetric groups and permutations \cite{TL321perm}.
In \cite{BaurKingMarsh} Baur et al. used Scott's map \cite{scott} to 
produce strand diagrams for triangulated surfaces, again with the 
same permutation, in each case. 
This raises the intriguing question of which permutations are
accessible in this way, and the role of the geometry in such
constructions. 
Strictly speaking, 
the precise identification of permutations is dependent, in this setup, on a
labelling convention. It is the numbers of permutations and the fibres
over them (as we investigate here)
that are, therefore, the main invariants accessible in the
present formalism.


\newcommand{\lscalee}{.4}

\begin{figure}
\includegraphics[scale=\lscalee]{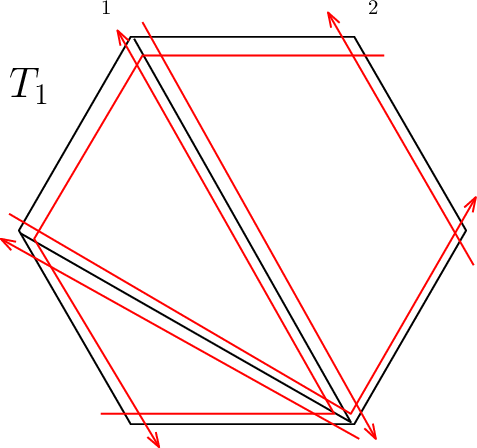}
\hskip .5cm
\includegraphics[scale=\lscalee]{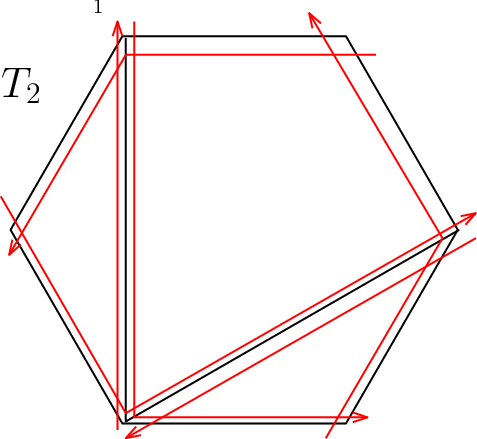}
\hskip .5cm
\includegraphics[scale=\lscalee]{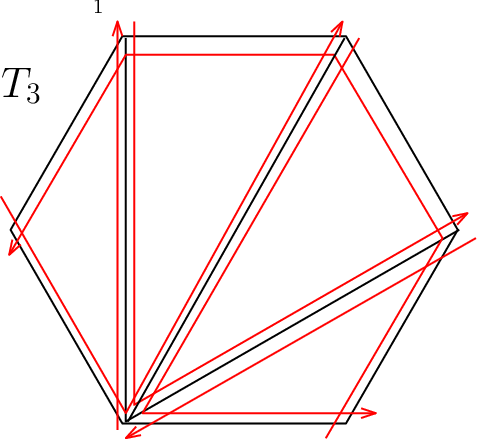}
$$
\Scott(T_1)=(146325)\hskip 1cm \Scott(T_2)=(142635)
    \hskip 1cm \Scott(T_3)=(135)(246)
$$
\includegraphics[scale=\lscalee]{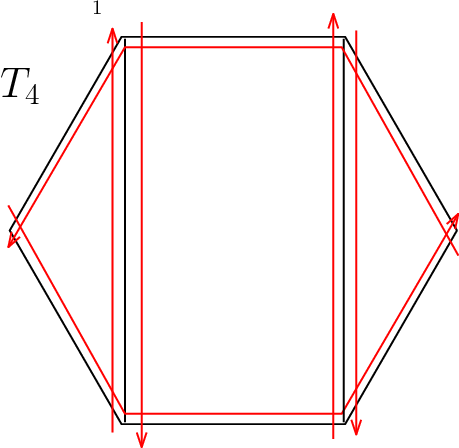}
\hskip .5cm
\includegraphics[scale=\lscalee]{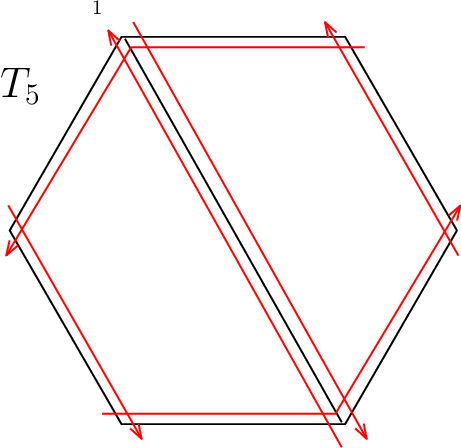}
\hskip .5cm
\includegraphics[scale=\lscalee]{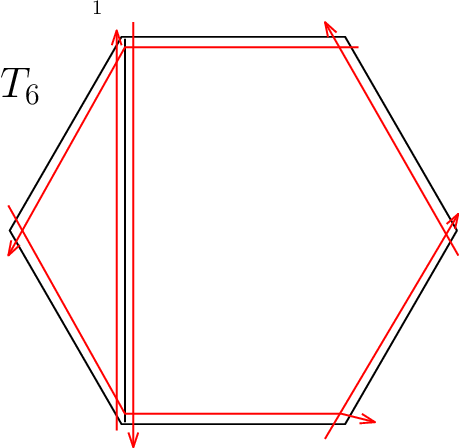}
$$
\Scott(T_4)=(15)(24)(36) \hskip 1cm \Scott(T_5)=(14)(2653) 
\hskip 1cm 
\Scott(T_6)=(15)(2643)
$$
\caption{Examples of tilings, 
  their strand diagrams
and permutations}\label{fig:hexagon}
\end{figure}


\section{Definitions and results} \label{sec:defq}

\newcommand{\Pu}{\operatorname{Pu}\nolimits} 
\newcommand{\Po}{\operatorname{Po}\nolimits} 
\newcommand{\Scottor}{{\mathbf \tau}}  
\newcommand{\umb}{umbral}
\newcommand{\Jor}{Jordan}
\newcommand{\medial}{\Jor}


\newcommand{\ka}{\kappa}

Given any manifold $X$ we write $\partial X$ for the boundary and $(X)$
for $X \setminus \partial X$. For a subset $D \in X$ we write
$\overline{D}$ for its closure \cite{Moise77}.

A {\em marked surface} is an oriented 2-manifold embedded in Euclidean
3-space, $S$; and  a finite subset $M$. 
Set $\Md = M \cap \partial S$.  
An {\em arc} in marked surface $(S,M)$ is a curve $\alpha$ in $S$ 
such that $(\alpha)$ is an embedding of the open interval in 
$(S)\setminus M$;
 $\partial \alpha \subseteq M$;
and if $\alpha$ cuts out a simple disk $D$ from $S$ then 
$|M \cap \overline{D}|>2$. 

Two arcs $\alpha,\beta$ 
in $(S,M)$
are {\em compatible} if there exist representatives 
$\alpha'$ and $\beta'$ in 
their isotopy classes such that $(\alpha')\cap(\beta')=\emptyset$. 

A {\em concrete tiling} 
of $(S,M)$ is a collection of pairwise compatible arcs that
are in fact pairwise non-intersecting. 
A {\em tiling} $T$ is a
boundary-fixing ambient isotopy class of concrete tilings
--- which we may specify by a concrete representative,
with arc set $E(T)$
(it will be clear that this makes sense on classes). 
A {\em tile} of tiling $T$ is a connected component of 
$S\setminus \cup_{\alpha\in E(T)} \alpha$.
We write $F(T)$ for the set of tiles. 
(Note that if $S$ is not homeomorphic to a disk then 
a tile need not be homeomorphic to a disk. 
For example a tile could be the whole of $S$ in the case of 
Figure~\ref{fig:non-polygonal}.)


\begin{figure}
\begin{center}
\includegraphics[scale=.6]{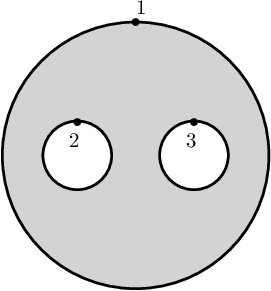}
\includegraphics[scale=.6]{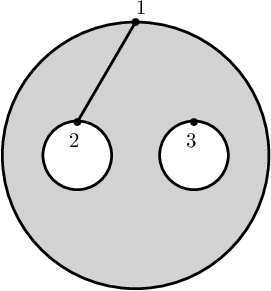}
\includegraphics[scale=.6]{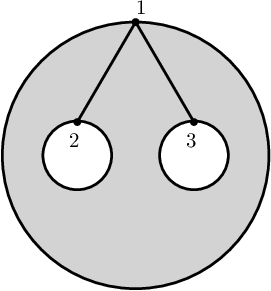}
\includegraphics[scale=.6]{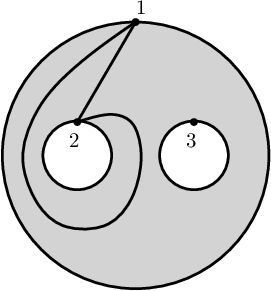}
\includegraphics[scale=.6]{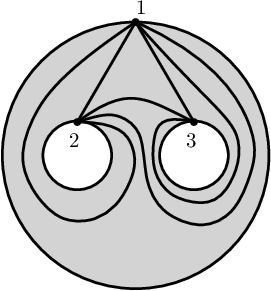}
\end{center}
\caption{
  Tilings of the pants surface.
  Here $\ka(S,M)=6\times0 +3\times 3+2\times 0+3-6=6$.}
\label{fig:non-polygonal}
\end{figure}


Fixing $(S,M)$, 
it is a theorem that 
there are maximal sets of compatible arcs. 
Set 
$\ka(S,M) = 6g+3b+2p+|M|-6$, 
where $g$ is the genus, $b$ the number of connected
components of $\partial S$,  
and $p=|M\cap(S)|$. 
Suppose 
$\ka(S,M) \geq 1$
and every boundary component intersects $M$.
Then $T$ maximal has   $|E(T)| = \ka(S,M)$,
and every tile is a simple disk bounded by three arcs.
Evidently given a tiling $T$ then the removal of an arc yields another
tiling. 
In this sense the set of tilings of $(S,M)$ forms a simplicial complex,
denoted $A(S,M)$.

We say two tilings are related by `flip' if they differ only by the
position of a diagonal triangulating a quadrilateral.
The transitive closure of this relation is called 
{\em \triangular\ equivalence}. 
We write $\te{T}$ for the equivalence class of the tiling $T$;
and $\AE(S,M)$ for the set of classes of $A(S,M)$.  


\subsection{The Scott\ map} \label{ss:scott}
\newcommand{\Pe}{P^{\pm}}
\newcommand{\Me}{M^{\pm}}

Let $L$ be a connected component of the boundary of an oriented
2-manifold, and $P$ a finite subset labeled 
$p_1, p_2,\dots, p_{|P|}$ in the
clockwise order
(a traveller along $P$ in the clockwise direction keeps the manifold
on her right). 
Then an {\em \umb\ set} 
$\Pe$ is a further subset of points $p_i^-$ and
$p_i^+$ ($i \in 1,2,\dots,|P|$) such that 
the clockwise order of all
these points is $\dots,p_i, p_i^+ , p_{i+1}^- , p_{i+1},\dots$.
That is, 
the interval $(p_i^- , p_i^+ ) \subset L$ contains only $p_i$.


Given $(S,M)$,
let $\Me$ denote a fixed 
collection of \umb\ sets over all boundary
components. 
A {\em \Jor\ diagram} $d$ on $(S,M)$ is
a finite number of closed oriented curves in $S$ together with 
a collection of 
$n=|\Md |$
oriented curves 
in $S$ such that each 
curve  
passes from some 
$p_i^+$ to some $p_j^-$ in $\Me$; 
and the collection of endpoints is $\Me$.
Intersections of curves are allowed, but must be transversal.
Write $\Scottor(d)$ for the permutation of $\Md$ this induces. 
That is, if $p_i^+$ goes to $p_j^-$ in $d$ then 
$\Scottor(d)(i)=j$.
Diagram $d$ is considered up to boundary-fixing isotopy. 
Let $\Pu(S,M)$ denote the set of \Jor\ diagrams. 


\begin{figure}
\[
\includegraphics[width=6.95cm]{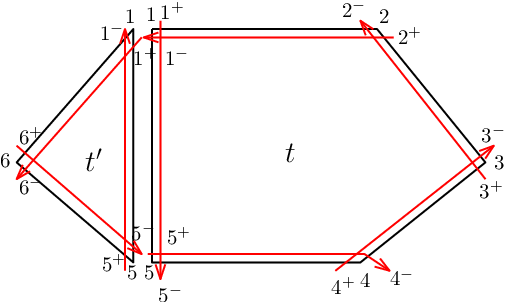}
\] 
\caption{Composing tiles and strand segments. 
Here $ \Scottor(d)(2)=6$.\label{fig:ctass}}
\end{figure}


Next we define a map $\Scottt :A(S,M) \rightarrow \Pu(S,M)$.
Consider a tiling $T $ in $A(S,M)$. 
By construction each  boundary $L$ of a tile $t$ is made up of
segments of arcs, terminating at a set of points $P$. 
%
Hence we can associate $P^{\pm}$ to $P$ 
as above. 
To  arc segment $s$ passing from $p_i$ to $p_{i+1}$ say,
we associate a {\em strand  segment} $\alpha_s$ 
in $t$ passing from $p_{i+1}^+$ to $p_i^-$,
such that the part of the tile on the $s$ side of 
strand segment $\alpha_s$ is a
topological disk. 
Finally strand segment crossings are transversal and minimal in number.
See tile $t$ in Figure~\ref{fig:ctass} for example.



It will be clear that if two tiles meet at an arc segment then the
\umb\ point constructions from each tile can be chosen to agree: 
as in Figure~\ref{fig:ctass}.
Applying the $\alpha_s$ construction to every segment $s$ of every
tile $t$ in $T$,  we thus obtain a collection 
$\Scottt(T)$ 
of strand segments in
$S$ forming strands 
whose collection of terminal points are 
at the \umb\ points of 
$\partial S$;
so that $\Scottt(T)\in\Pu(S,M)$. 
Altogether, writing $\Sym_M$ for the set of permutations of set $M$, 
we have $\Scott : A(S,M) \rightarrow \Sym_{\Md}$ defined by
\begin{equation} 
\label{eq:compost}
\Scott = \Scottor \circ \Scottt
\end{equation}
We call this the {\em Scott map}. 
It agrees with Scott's construction \cite{scott}
in the case of triangulations of simple polygons. 

We remark that the intermediate map $\Scottt$ is injective, 
as we will show later (Theorem~\ref{th:th3}). 
The map $\Scott$ however is clearly not injective, as the image of any triangulation 
of a polygon is the permutation induced by $i\mapsto i+2$. 





\medskip


The  focus of this article is the case where 
$(S,M)$ 
is a polygon $P$ with $n$ vertices. 
We  write $A_n$ for 
$A(S,M)$ in this case,
$\AE_n$ for $\AE(S,M)$,
and $\Pu_n$ for $\Pu(S,M)$. 
Our main result can now be stated: 

\begin{theorem}\label{thm:tilings-permut-inj} \label{th:main}
Let $T_1, T_2$ $\in A_n$ be tilings of an $n$-gon $P$. 
Then 
$\Scott(T_1)=\Scott(T_2)$ if and only if 
$\te{T_1} = \te{T_2}$.
\end{theorem}


Sections \ref{sec:machinery}, \ref{sec:pauls-sec-0}
and \ref{sec:2nd-machinery}, \ref{sec:subdiv-to-permut}
are concerned with the proof of this result. 


We will see in Lemma~\ref{lm:pf22}
that $\Scottt(A_n)$ lies in the subset of $\Pu_n$ of
alternating strand diagrams \cite[\S14]{Postnikov}.
Theorem~\ref{th:main} is
thus related to 
Postnikov's result~\cite[Corollary 14.2]{Postnikov} 
that the permutations 
arising from two alternating strand diagrams are the same if and only
if the strand diagrams  
can be obtained from each other through a sequence of 
certain kinds of `moves'. 
Consider the effect of a flip on
the associated strands: 
\[
\includegraphics[scale=.5]{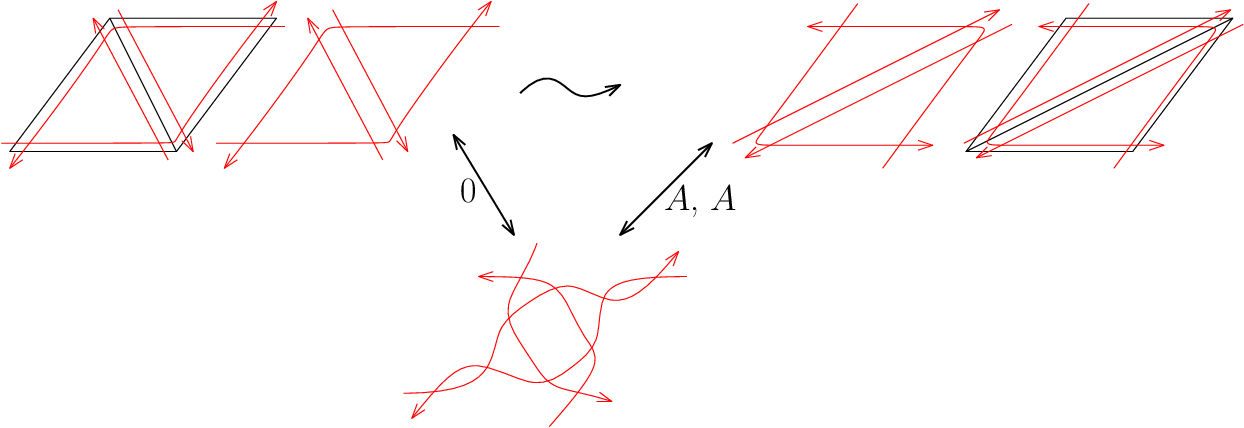}
\]
Comparing with Figure 14.2 of~\cite{Postnikov},
 the diagram
shows that the flip corresponds to a certain combination of
two types of Postnikov's three moves (see Figure~\ref{fig:sdmoves1}).




\newcommand{\psiz}[1]{\mbox{\small{#1}}}
\begin{figure}
\[
\raisebox{-.2031in}{
\includegraphics[width=6cm]{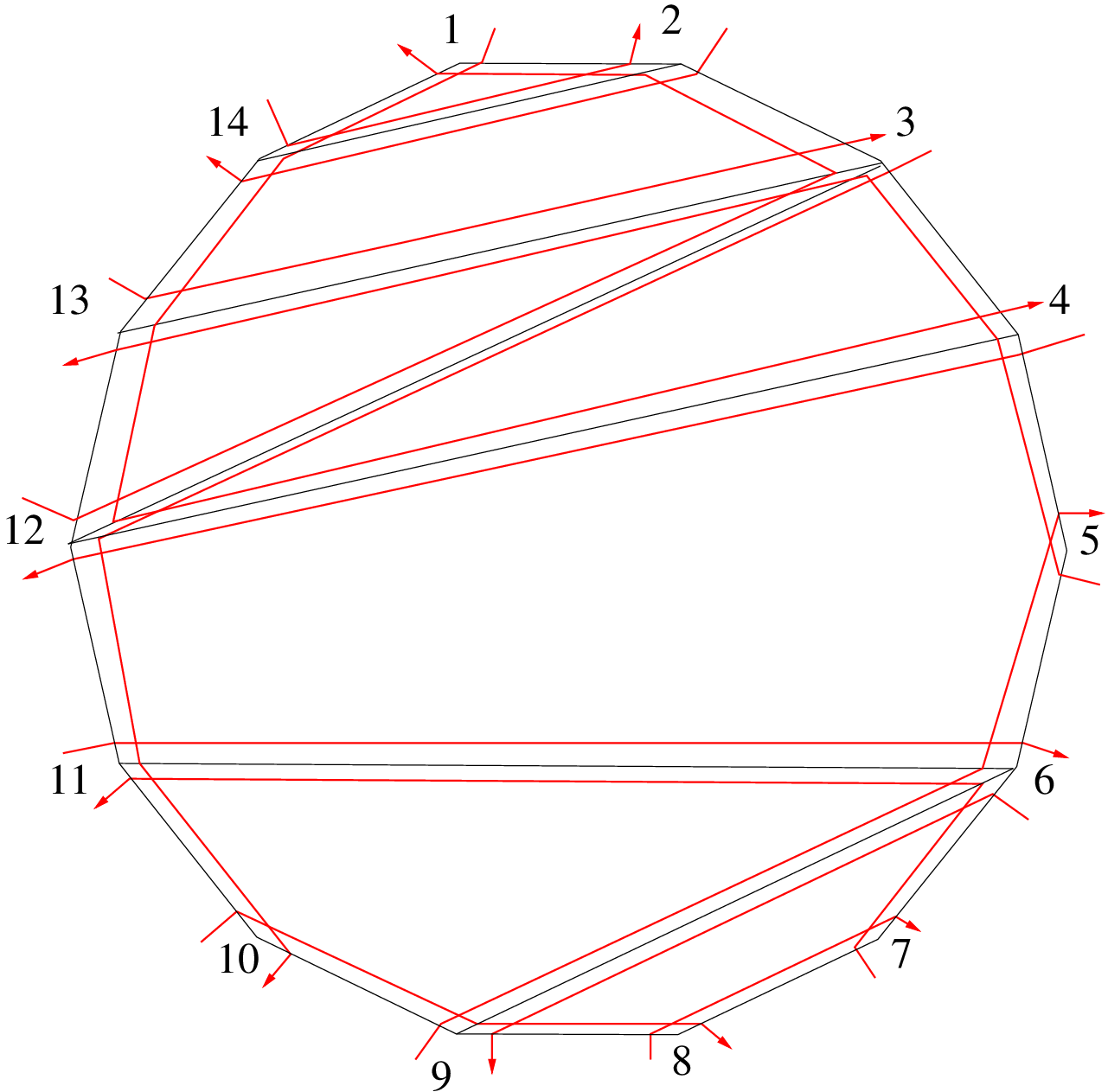}}
\hspace{.3in}
\raisebox{-.2031in}{
\includegraphics[width=6cm]{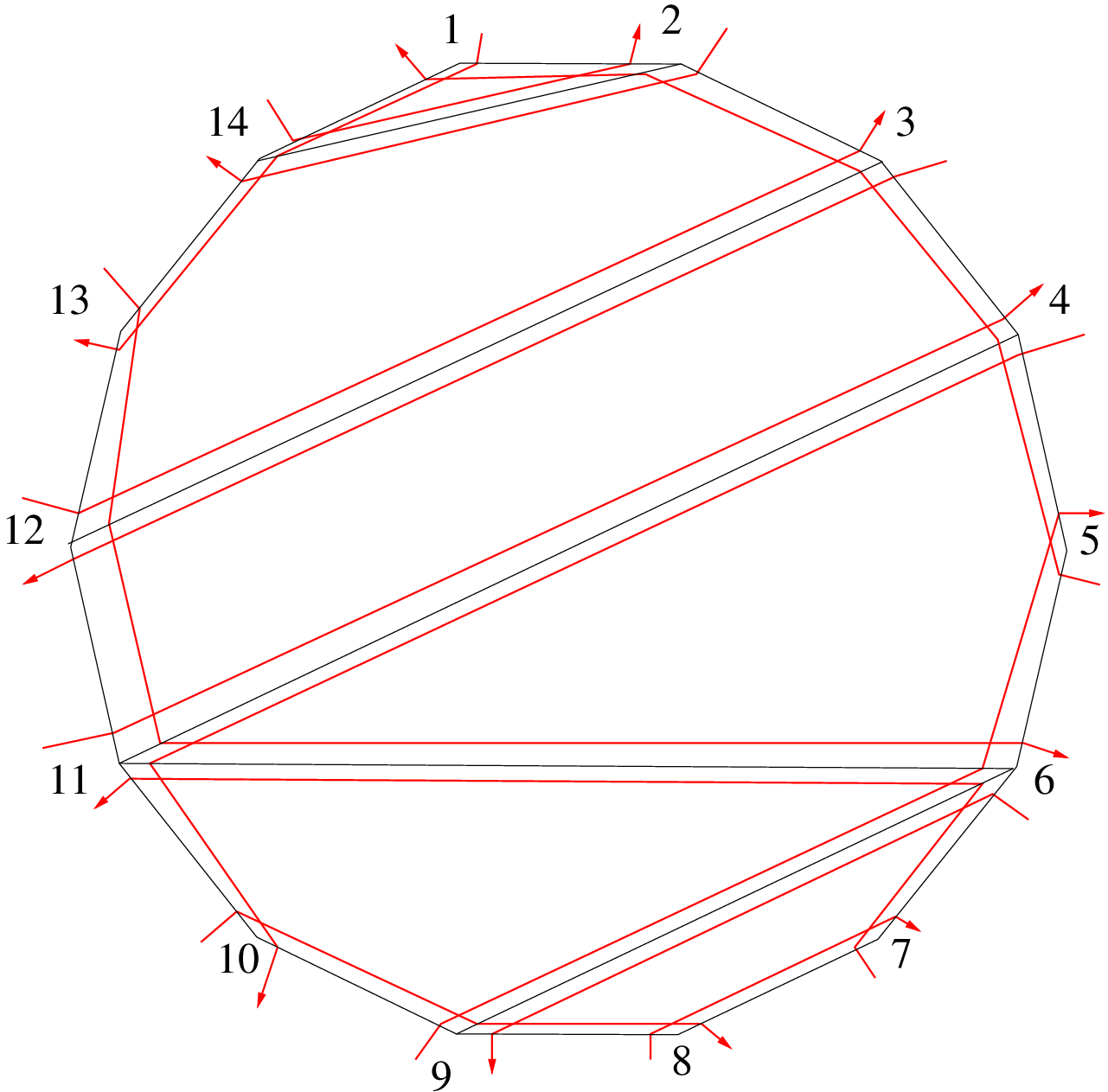}}
\]\[
\hspace{.1in}
\stackrel{\Scott}{\mapsto}
\psiz{(1,4, 12)(2, 14)(3,10,8,7,11,6,9,5,13)}
\hspace{.55in}
\stackrel{\Scott}{\mapsto}
\psiz{(1,13,6,9,5)(2,14)(3,12)(4,10,8,7,11)}
\]
\caption{Examples of tilings, strands and Scott maps
\label{fig:p14}}
\end{figure}


\subsection{Notation for tilings of polygons} 
We note here simplifying features of the polygon case
that are useful in proofs. 


Geometrically we may consider a tile as a subset of polygon $P$
considered as a subset of $\RR^2$. 
This facilitates the following definition. 

\newcommand{\subdivision}{tiling}
\newcommand{\RD}{R_{12}} 
\newcommand{\Tr}{\operatorname{Tr}\nolimits} 
\newcommand{\tp}{triangulated-part}

\begin{defn}\label{def:triangle-part}
Let $T \in A_n$.
By $\Tr(T) \subset \RR^2$ 
we denote the union of all triangles in $T$.
We call $T_1$ and $T_2$ {\em \tp\ equivalent}
if $\Tr(T_1)=\Tr(T_2)$ and they agree on the complement of $\Tr(T_1)$. 
See Figure~\ref{fig:triangle-parts}.
\end{defn}



\begin{figure}
\includegraphics[scale=.6]{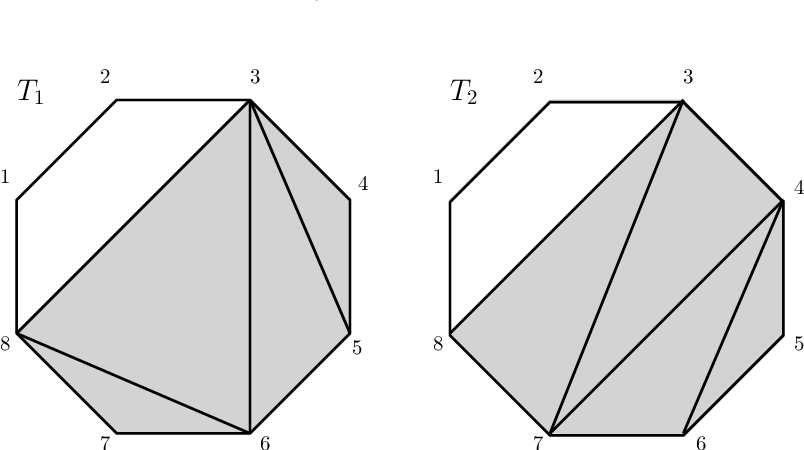} 
\includegraphics[scale=.6]{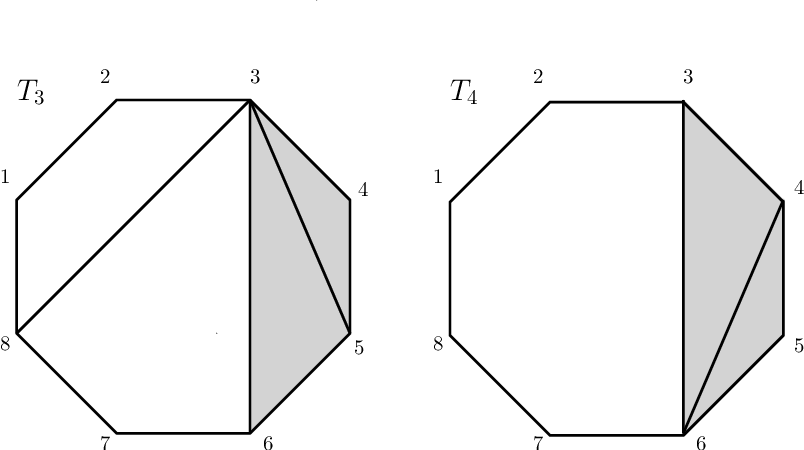} 
\caption{
Tilings of an octagon, and 
the associated $\Tr(T_i)$. Note that 
$T_1,T_2$ are \tp\ equivalent, but $T_3,T_4$ are not.
}\label{fig:triangle-parts}
\end{figure}



By Hatcher's Corollary in~\cite{Hatcher}, 
two tilings are flip equivalent if and only if 
they are \tp\ equivalent. 
The following is immediate.
\begin{lm}
Let $T_1$ and $T_2$ be tilings of an $n$-gon $P$. 
All tiles of size $\ge 4$   
agree in these tilings if and only if
$\te{ T_1} = \te{ T_2 }$.
\end{lm}
 

\newcommand{\di}[1]{[ #1 ]} 

Write $\ul{n}=\{1,2,\dots, n\}$ for the vertex set of $P$, 
assigned to vertices as for example in Figure~\ref{fig:p14}.
The `vertices' of $A_n$ {\em as a simplicial complex} are the 
$n(n-3)/2$ diagonals.
A diagonal between polygon vertices $i,j$ is uniquely determined by
the vertices. We write $\di{i,j}$ for such a diagonal. 
Here order is unimportant.
A tiling in $A_n$ can then be 
given as its set of diagonals.
Example: The tiling from $A_8$ 
in Figure~\ref{fig:first1} is 
$T = \{\di{2,8}, \di{3,5}, \di{5,8} \}$.
An example of a top-dimensional simplex (triangulation) in $A_8$ of which this
$T$ is a face is 
$T \cup \{ \di{3,8}, \di{6,8}\}$.

Equally usefully, focussing instead on tiles, 
we may represent a tiling $T \in A_n$ as a subset of 
the power set  $\Power(\ul{n})$:
for $T \in \Power(\ul{n})$ one includes the subsets that are  
the vertex sets of tiles in $T$. 

\begin{ex}
In tile notation 
the tiling from $A_8$ in Figure~\ref{fig:first1} becomes 
\[
T = \{ \{1,2,8 \},\{2,3,5,8 \},\{3,4,5 \}, \{ 5,6,7,8 \} \}
\]
In this representation, while $A_2 = \emptyset$, 
and $ A_3 = \{ \{ \{ 1,2,3 \} \} \}$,
we have:
\[
A_4 = \{ \{ \{ 1,2,3,4 \}           \}, \;
                \{ \{ 1,2,3 \}, \{1,3,4 \} \} , \; 
                \{ \{ 1,2,4 \}, \{2,3,4 \} \} \}
\]
\end{ex}



We present two proofs of Theorem~\ref{th:main}: 
one by constructing an inverse 
--- in Section~\ref{sec:pauls-sec-0} we show 
how to determine 
the \triangular\ equivalence class from the permutation; 
and one by direct geometrical arguments 
--- see Section~\ref{sec:subdiv-to-permut}.
We first establish 
machinery used by both.



\section{Machinery for proof of Theorem~\ref{th:main}}  
\label{sec:machinery} \label{ss:tree1}


The {\em open dual} 
$\gam(T)$ of tiling $T$ is the dual graph of $T$ regarded as a
plane-embedded graph (see e.g. \cite{BondyMurty})
excluding the exterior face (so restricted to vertex set
$T$). 
See Figure~\ref{fig:treeeg} for an example. 

\begin{lm} \label{lm:tree}
Graph $\gam(T)$ is a tree.
\end{lm}
\begin{proof}
By construction the boundary of plane-embedded graph 
$\gam(T)$ (see e.g. \cite{MarshMartin}) 
has the same number of components as the boundary of $T$.
\end{proof}

A {\em proper tiling} is a tiling with at least two tiles. 
An {\em ear} in a proper tiling $T$ is a tile with one edge a diagonal. 
An $r$-ear is an $r$-gonal ear.

\begin{cor} \label{lm:graphear}
Every proper tiling has at least 2 ears.
\qed
\end{cor}

\begin{figure}
\includegraphics[width=2.5in]{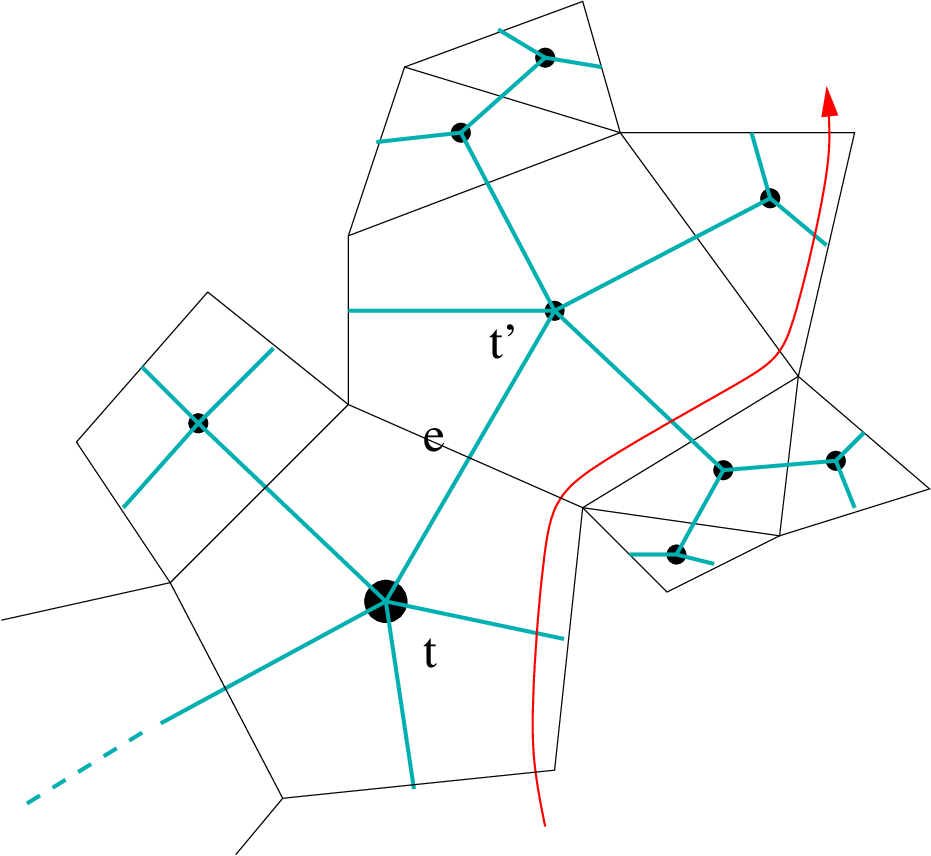}
\caption{\label{fig:treeeg} Dual tree example.}
\end{figure}


\subsection{Elementary properties of strands}

Consider a tiling $T$. 
Note that a tile $t$ in $T$ and an edge $e$ of $t$ 
determine a strand of 
the $\Scott(T)$ construction --- the strand leaving $t$ through $e$. 
Now, when a strand $s$ leaves a tile $t$ through an edge $e$ 
it passes to an adjacent tile $t'$ (as in Figure~\ref{fig:treeeg}), 
or exits $P$ and terminates.
We associate a (possibly empty) branch $\gamma_{t,e}$ of $\gamma(T)$ to this
strand at $e$: the subgraph accessible from 
the vertex of $t'$ without touching
$t$. 
Note that the continuation of the strand $s$ leaves $t'$ at some  edge
$e'$ distinct from $e$, 
and that $\gamma_{t',e'} $ is a subgraph of $ \gamma_{t,e}$. 

\begin{lm} \label{lm:noreturn}
Consider the strand construction on a tiling. 
After leaving a tile at an edge, a strand does not 
return to cross the same edge again.
\end{lm}
\begin{proof}
Consider the strand as in the paragraph above. 
If the strand exits the polygon $P$ at $e$ we are done. Otherwise, 
since the sequence of graphs $\gamma_{t^i,e^i}$
associated to the passage
of the strand is a decreasing sequence 
of graphs, containing each other, it 
eventually leaves $P$ and terminates in some tile of a vertex of 
$\gamma_{t',e'}$ and so does not return to $t$. 
\end{proof}

An immediate consequence of Lemma~\ref{lm:noreturn} is the following: 
\begin{cor}\label{cor:only-one-strand}
A strand of a tiling $T$ can only use one strand segment of a given
tile of $T$. 
\end{cor}

\begin{lm} \label{lm:pf22}
Let $T\in A_n$ be a tiling of 
an $n$-gon. Then the strands of 
$\Scottt(T)$ have the following properties \cite[\S 14]{Postnikov}: $\ $
(i) Crossings are transversal and the strands crossing a given 
strand alternate in direction.  
(ii) If two strands cross twice, they form an oriented digon. 
(iii) No strand crosses itself. 
\end{lm}

\begin{proof}
The first two properties follow from the construction. That no strand crosses 
itself follows from Corollary~\ref{cor:only-one-strand}. 
Note that the underlying polygon can be drawn convex, in which case
strands are left-turning.
The requirement  that there are no unoriented lenses 
follows from the fact that strands
are left-turning in this sense.  
(Remark: our main construction is unaffected  by non-convexity-preserving
ambient isotopies, but the left-turning property is only preserved
under convexity preserving maps.)
\end{proof} 


\medskip



We write $x \leadsto y$ for a strand starting at vertex $x$ and ending
at vertex $y$. Thus if  $x \leadsto y$ is a strand of tiling $T$ and 
$\sigma$ is $\Scott(T)$ then this strand determines $\sigma(x)=y$.

If a list of vertices is ordered minimally clockwise around the polygon, we will often just 
say clockwise, for example $(7,1,2)$ is ordered minimally clockwise. 
To emphasise that vertices $x_1,x_2,x_3$ 
are ordered minimally clockwise, 
we will repeat the ``smallest'' element at the end: 
$
x_1<x_2<x_3<x_1.
$ 

\begin{defn}
Let $q$ be a vertex of a polygon with strand diagram. \\
(1) We say that a strand $x\leadsto y$ {\em covers $q$} if 
we have $x<q<y<x$ minimally clockwise. \\
(2) We say that $x\leadsto y$ {\em covers} strand $x'\leadsto y'$ if 
$x<x'<y'<y<x$ or $x<y'<x'<y<x$. 
\end{defn}


\subsection{Factorisation Lemma}


$\quad$

Consider the two strands $s_1,s_2$ passing through an edge $e$ of a tile $t$.
We say these strands are `antiparallel at $e$'; and consider the 
`parallel' strand $s_1$ and antistrand
$\overline{s_2}$ both moving into
$t$ from $e$. 
Examples:

\begin{center}
\includegraphics[width=4cm]{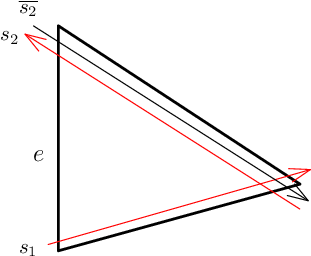}
\hskip 1cm
\includegraphics[width=4cm]{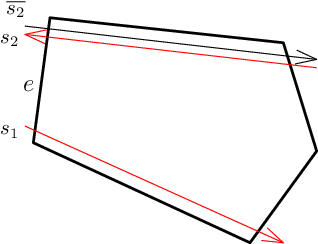}
\end{center}

\begin{lm} 
[`Lensing Lemma']
\label{lm:lensing}
(I)
Let strand segments $s_1$ and $s_2$ be antiparallel at an edge $e$ of a tile $t$
in a polygon tiling $T$. 
Traversing the two segments in the direction 
from $e$ into the tile $t$, they do one of the following: 
(a) if $t$ is a triangle the segments cross in $t$ and do not meet again;
(b) if $t$ is a quadrilateral the segments leave $t$ antiparallel in the opposite
edge; 
(c) if $|t| >4$ they leave $t$ in different edges and the strands do not cross
thereafter.  
\\
(II)
In any polygon tiling $T$, two strands cross at most twice. If two
strands cross twice then (i) they pass through a common edge $e$;
(ii) the crossings occur in triangles, on either side of $e$,
with only quadrilaterals between.
\end{lm}

\begin{proof}
(I) See the figure. 
Note that in cases (a) and (c) the strands pass out
of $t$ through different edges and hence into different subpolygons. 
Now use Lemma~\ref{lm:noreturn}. 
(II) Every crossing has to occur in a tile. If two strands enter a tile 
across different edges, they have not crossed before entering into the tile 
(Lemma~\ref{lm:tree}). The claim then follows from (I). 
\end{proof}



\begin{lm}\label{lm:permutations-not}\label{lm:baur-magic}
Let $T$ be a tiling of an $n$-gon and $\sigma=\Scott(T)$. 
Then (a) $\sigma$ 
has no fixed points and (b) there is no $i$ with $\sigma(i)=i+1$. 
\end{lm}

\begin{proof} (a) Follows from the left-turning property
(cf. Proof of Lemma~\ref{lm:pf22}). 
(b) 
Consider the tile with edge $e=[i,i+1]$.
The strand starting at $i$ and the strand ending 
at $i+1$ have segments in the same tile and hence differ by Corollary 
\ref{cor:only-one-strand}.  
\end{proof}



\begin{lm}
[`Factorisation Lemma']
\label{lm:glue-tiles}
Let $P$ be a polygon and $T_1,T_2$ two tilings of $P$. Assume that there exists 
a diagonal $e=[i,j]$ in
$T_1$  and $T_2$. 
Denote by $P'$ the polygon on vertices $\{i,i+1,\dots, j-1,j\}$.
We have: \\
$ \Scott(T_1) = \Scott(T_2)
\Longrightarrow
\Scott(T_1|_{P'}) = \Scott(T_2|_{P'})
$.
\end{lm}

\begin{proof}
Consider Figure~\ref{fig:gluing}. 
The only way a strand of $\Scottt(T_1)$ passes out of $P'$ 
is through $e$, and there is exactly one such strand (and one passing in). 
This strand is non-returning by Lemma~\ref{lm:noreturn}, 
so 
its endpoints are identifiable from $\sigma=\Scott(T_1)$ as the unique
vertex pair $k,l$ with $\sigma k = l$ and with $k$ in $P'$ and $l$ not. 
Apart from this and the corresponding `incoming' pair with 
$\sigma k' = l'$,
all other strand endpoint pairs of 
$\Scottt(T_1|_{P'})$ are as in $\Scott(T_1)$
and hence agree with $\Scottt(T_2|_{P'})$
if $\Scott(T_1)=\Scott(T_2)$. 
Indeed, if $\Scott(T_1)=\Scott(T_2)$ then $\Scott(T_2)$ identifies the same
two pairs $k,l$ and $k',l'$. 
At this point it is enough to show that the image of vertex $k$ under
$\Scottt(T_1|_{P'})$, which is either vertex $i$ or $j$, is determined by
$\sigma$
(since the determination will then be the same for $\Scottt(T_2|_{P'})(k)$). 
If $k \geq l'$ then the strands $k \leadsto l$ and $k' \leadsto l'$ 
cross over each other in $P'$, 
since otherwise they have the wrong orientation at $e$ 
--- see Fig.\ref{fig:gluing}. 
Thus the image of $k$ is determined (it is $i$). 
If $k < l'$ then the image is determined similarly 
(this time non-crossing is forced and the image is $j$). 
\end{proof}

\begin{figure}
\includegraphics[width=10cm]{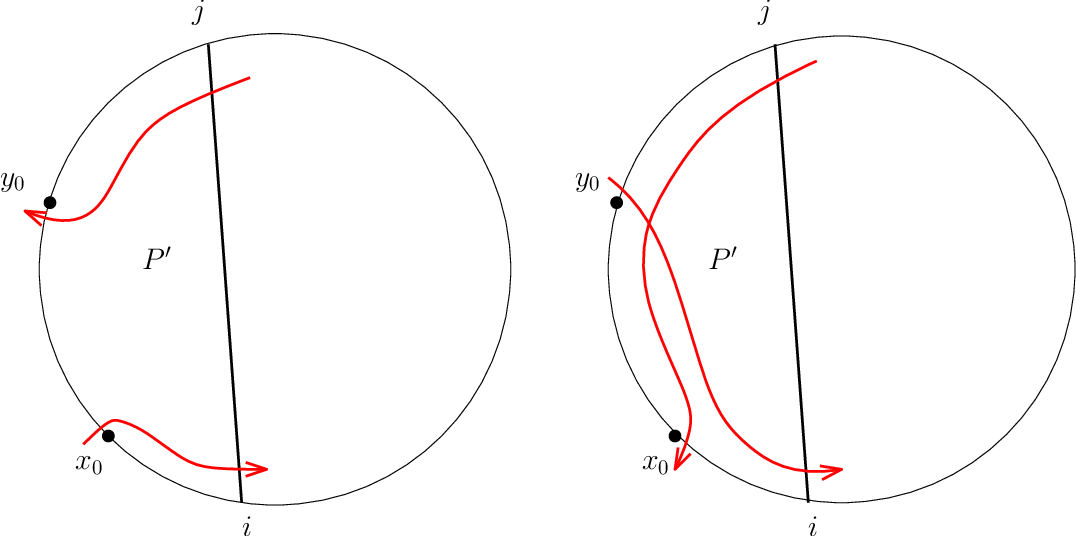}
\caption{Schematic for 
two strands passing through diagonal $e=[i,j]$}\label{fig:gluing}
\end{figure}

\subsection{Properties of strands and tiles} 


$\quad$

We will say that a vertex in polygon $P$ is {\em simple} in tiling $T$ if it
is not the endpoint of a diagonal. 
We will say that an edge $e=[i,i+1]$
of $P$ is a {\em simple edge} in $T$ if both vertices are
simple. 

\begin{lm}\label{lm:simple-edge-rotation}
A strand $i+1 \leadsto i$ arises in $\Scott(T)$ if and
only if the edge $[i,i+1]$ is simple. 
\end{lm}
\begin{proof}
If $[i,i+1]$ is simple in $T$, the claim follows by construction. 
If $i$ is not simple, then the strand ending at $i$ contains 
the strand segment of the diagonal $[j,i]$ of $T$ with $j<i-1$ maximal clockwise 
and has starting point in $\{j,j+1,\dots, i-2\}$. 
Similarly, if $i+1$ is not simple, the strand starting at $i+1$ 
contains the strand segment of the diagonal $[i+1,k]$ with $k>i+2$ minimal 
anticlockwise. Its ending point is among 
$\{i+3,\dots, k\}$. 
\end{proof} 



\begin{figure}
\includegraphics[width=4.2cm]{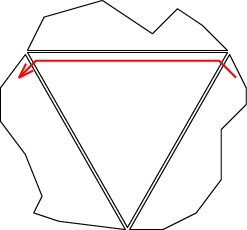}
\includegraphics[width=4.2cm]{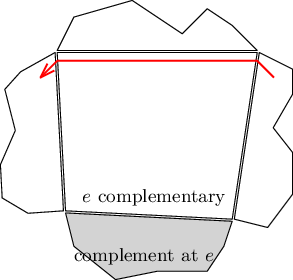}
\includegraphics[width=4.6cm]{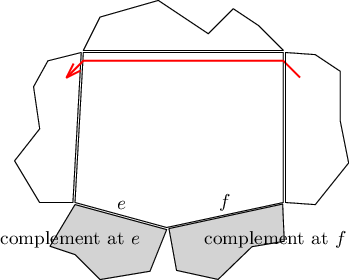}
\caption{Tiles with complementary edges and their complements
\label{fig:complement3-5}}
\end{figure}

In general a strand passes through a sequence of tiles. At each such
tile it is parallel to one edge and passes through the two adjacent
edges. Any remaining edges in the tile 
are called {\em complementary} to the
strand. Each of these complementary edges defines a sub-tiling --- the
tiling of the part of $P$ on the other side of the edge. We call this 
the {\em complement to the corresponding edge}. 
Note that the strand covers every vertex 
in this sub-tiling.
See for example Figure~\ref{fig:complement3-5}. 
We deduce:

\begin{lm} \label{lm:strand-tile-magic}
(I) A strand $i \leadsto i+2$ passes only through
  triangles. \\
(II) A strand $i \leadsto i+3$ passes through one quadrilateral
(with empty complement) and
  otherwise triangles. \\
(III) A strand $i \leadsto i+4$ passes through one quadrilateral
(with a complementary triangle)
or two quadrilaterals or
  one pentagon (with empty complement), and otherwise triangles. \\
(IV) A strand $i \leadsto i+k$ passes through a 
tile
sequence $Q_i$ such
  that 
$$
k-2 \geq \sum_i \left(|Q_i| -3\right)
$$
(the non-saturation of the bound corresponds to some tiles having
non-empty complement).
\qed
\end{lm}

\begin{ex}\label{ex:complements}
As an illustration for Lemma~\ref{lm:strand-tile-magic} consider Figure~\ref{fig:p14}. 
Both tilings have a strand $6\leadsto  9$, illustrating the case $k=3$. \\
In the tiling on the left, there is a strand $13\leadsto 3$, it passes through one 
quadrilateral with complementary triangle $\{14,1,2\}$. 
\end{ex}



\begin{lm} \label{lm:tear}
Let $T \in A_n$ and $\sigma=\Scott(T)$. 
Then $\sigma (i) = i+2$ if and only if there exists $T' \in \te{T}$ 
with a 3-ear at vertex $i+1$. 
\end{lm}
\begin{proof}
If: $ T'$  has a strand direct from $i$ to $i+2$ in the given ear.
Note that 
all other tilings in $\te{T'}$ have only triangles incident at $i+1$
(since a neighbourhood of $i+1$ lies in the triangulated part). One
sees from the construction that these tilings all
have a strand from $i$ to $i+2$.
See (a):
\[ 
(a)
\raisebox{.237in}{\includegraphics[width=1in]{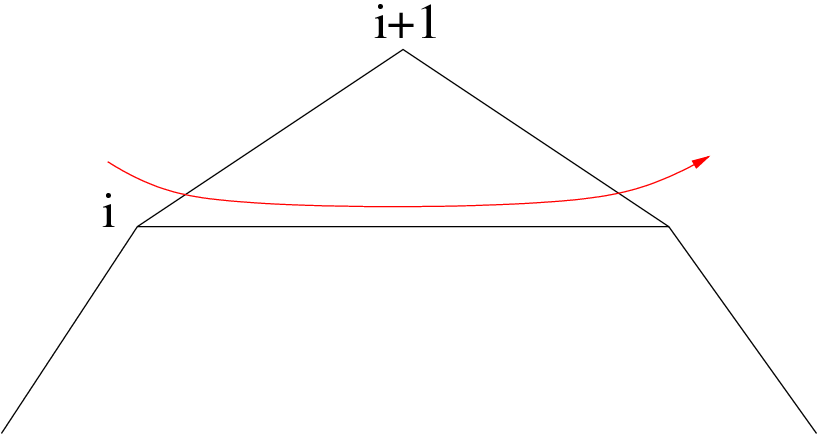}}
\quad
\includegraphics[width=1in]{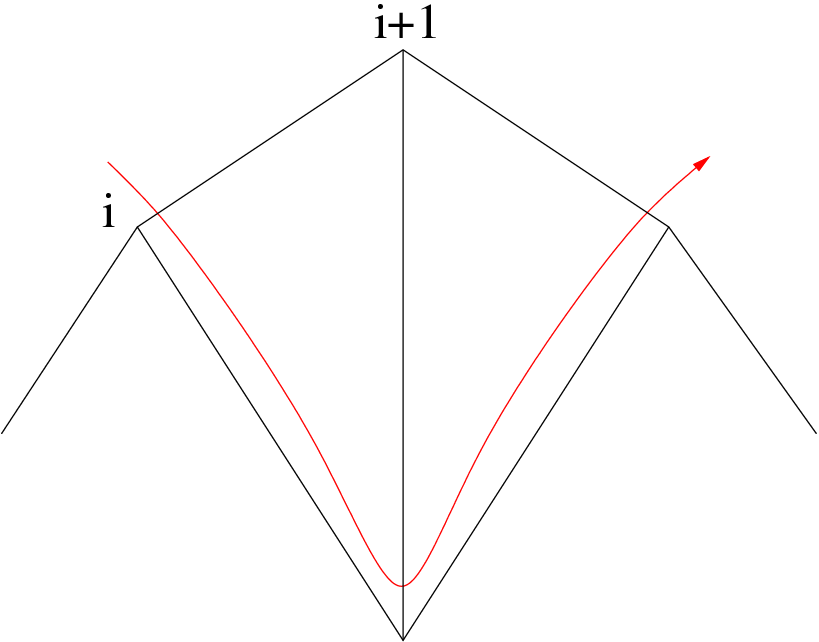}
\quad
\includegraphics[width=1in]{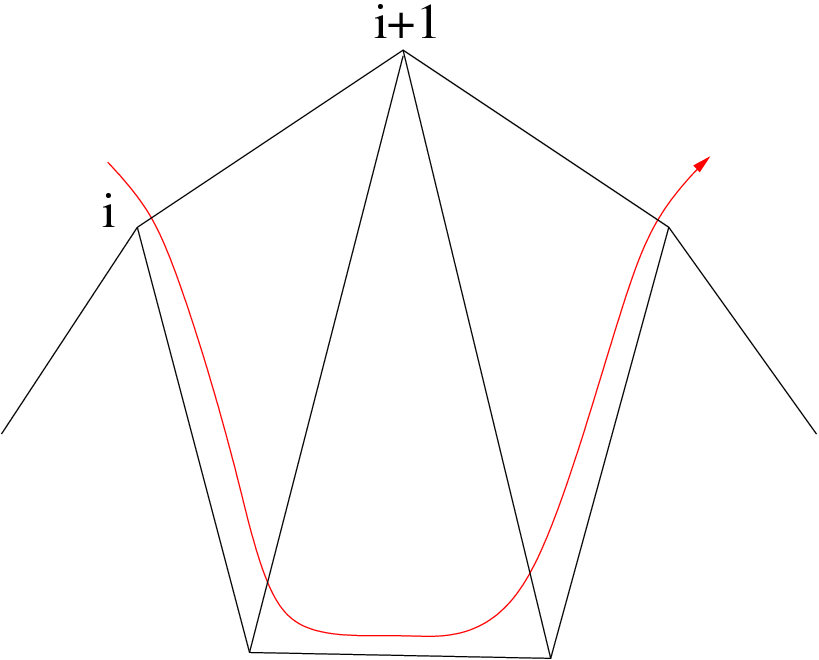}
\hspace{.74in}
(b)
\raisebox{-.112in}{\includegraphics[width=1in]{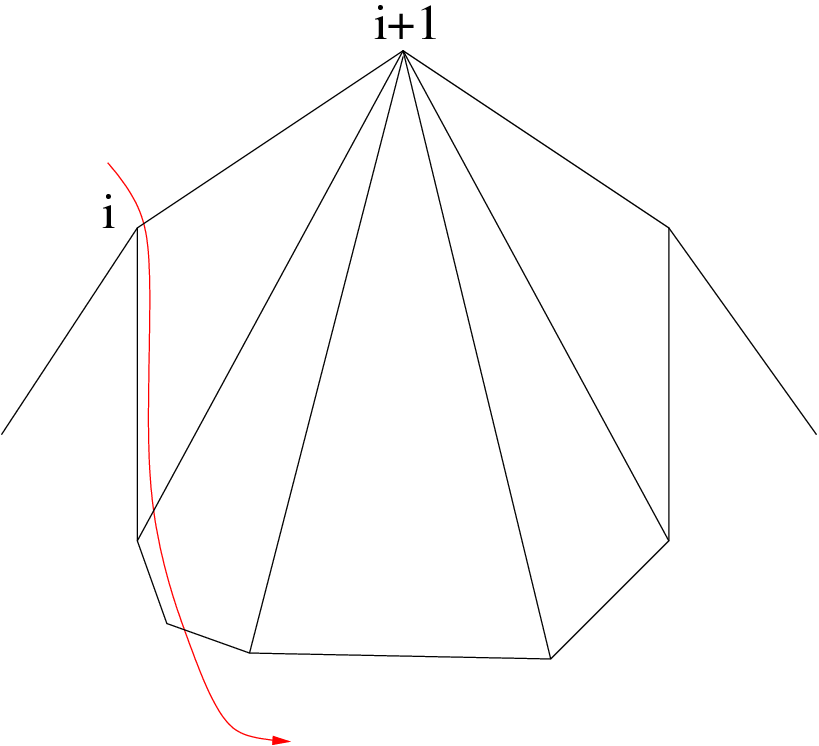}}
\]
Only if: If there is no such $T'$ in $\te{T}$ then among the tiles incident at
$i+2$ is one with order $r>3$. The strand from $i$ passes into $P$ at the
first tile incident at $i+1$. 
If this is a triangle then the strand
passes into the second tile, and so on. Thus eventually 
the strand meets a tile of higher order --- see (b) above. 
But then by Lemma~\ref{lm:strand-tile-magic} 
we have $i \leadsto i+k$ with $k>2$. 
\end{proof}



For given $n$ let us write $\ttt$ for the basic cycle element in $\Sym_n$:
$\ttt=(1,2,...,n)$.
The following is implicit in \cite{scott},
and is a corollary to Lemma~\ref{lm:strand-tile-magic}. 
\begin{lm}\label{lm:triangle-case}
Let $T$ be an arbitrary triangulation of an $n$-gon. Then the permutation 
$\Scott(T)$  associated 
to $T$ is induced by $i\mapsto i+2$, $1\le i\le n$ (reducing mod
$n$). 
Indeed, for $T \in A_n$, 
$\Scott(T) = \ttt^2$ if and only if $T$ is a triangulation.
\end{lm}



\begin{defn}
A {\em run} is a subsequence of form $i-1,i-2,...,i-r+1$ in a cycle of a 
permutation of $S_n$. 
A maximal subsequence of this form is an {\em $r$-run at $i$}.
\end{defn}

In Figure~\ref{fig:p14}, both permutations have a $3$-run at $9$. 

\begin{lm} \label{lm:bigear}
Let $T\in A_n$, and $\sigma=\Scott(T)$. We have \\
i) 
$\sigma$ contains a cycle of length $\ge r$, where $r\ge 2$, 
with 
an $r$-run at $j$
$\Longleftrightarrow$ 
$[j-1,j-2]$, $[j-2,j-3], \dots, [j-r+2,j-r+1]$ 
is a maximal sequence of simple edges in $T$; 
\\ ii) 
Assume $\Scott(T)$ is as in (i) and $r<n-1$. 
Then TFAE \\
(a) $[j-r,j]\in T$; 
(b) $\{j-r,j-r+1,\dots, j\}$ is an $(r+1)$-ear in $T$; 
(c) $\sigma(j-r)=j$. 
\end{lm}

Note that the case $r=2$ occurs if $j-1$ is simple, while the edge $[j-1,j-2]$ is not simple - a triangular ear. 

\begin{proof}

i) The implication $\Longleftarrow$ follows by construction. Implication $\Longrightarrow$ follows 
from Lemma~\ref{lm:simple-edge-rotation}. 

ii) Observe that the the assumptions in ii) are consistent with (b). \\
(a) $\Longrightarrow$ (b) follows with i). (b) $\Longrightarrow$ (c) follows from the 
construction. \\
To show (c) $\Longrightarrow$ (a) first note that by the assumptions, $j$ and $j+r$ are 
not simple. Among the diagonals incident with $j$ consider the diagonal $[j,q_1]$ with 
endpoint $q_1$ maximal (and clockwise) from $j$. 
Among the diagonals incident with $j-r$ consider 
the diagonal $[q_2,j-r]$ with $q_2$ minimal (and anticlockwise) from $j-r$. 
If $q_1=j-r$ (and 
hence $q_2=j$), we are done. 
So assume for contradiction that $j<q_1\le g_2<j-r<j$. Both diagonals are 
edges of a common tile $Q$ containing the simple edges 
$[j-1,j-2]$, $[j-2,j-3], \dots, [j-r+2,j-r+1]$. 
Consider 
the strand starting at $j-r$. It leaves the tile $Q$ in $\{q_1+1,\dots, q_2\}$. 
By Corollary~\ref{cor:only-one-strand} 
it cannot return back into $Q$, and so its 
endpoint is different from $j$. 
\end{proof}


\section{
Inductive proof of Theorem}\label{sec:pauls-sec-0}
\medskip

One proof strategy for the 
main theorem (Theorem~\ref{th:main}) is as follows.
We assume the theorem is true for \rank s $m<n$ (the induction base is clear).

The `If' part follows from the Factorisation Lemma (Lemma~\ref{lm:glue-tiles}) and 
Lemma~\ref{lm:triangle-case}.

For the `Only if' part proceed as follows. 
Consider $T_1,T_2$ with $\sigma=\Scott(T_1) = \Scott(T_2)$.
Note that $T_1$ has an ear, either triangular or
bigger (Lemma~\ref{lm:graphear}). Pick such an ear $E$. 
Consider the cases (i) $|E|=3$; (ii) $|E|\neq 3$.
\\
(i) If $E $ is triangular in $T_1$ then $\sigma=\Scott(T_1)$ has 
$i \leadsto i+2$ at the corresponding position.
Thus so does $\Scott(T_2)=\Scott(T_1)$, and hence there is a $T_2'$ in
$\te{T_2}$ also with this ear, by Lemma~\ref{lm:tear}. 
Note that $\Scott(T_2')=\Scott(T_2)$ 
since $T_1\sim_{\triangle} T_2$. 

Since $T_1\setminus E$ and $T_2'\setminus E$ are well defined we have
$\Scott(T_1\setminus E) = \Scott(T_2'\setminus E)$
by the Factorisation Lemma (Lemma~\ref{lm:glue-tiles}). 
That is, the Scott permutations $\Scott(T_1)$ and 
$\Scott(T_2')$
of $T_1$ and $T_2'$ agree on the part
excluding this triangle. 
But then $\te{(T_1\setminus E)} = \te{(T_2'\setminus E)}  $
(i.e. the restricted tilings agree up to triangulation) by the 
inductive assumption. 
Adding the triangle back in we have $\te{T_1} = \te{T_2'}$.
But $\te{T_2'} = \te{T_2}$ and we are done for this case.
\\
(ii)
If ear $E$ is not triangular in $T_1$ then 
$T_2$ has an ear in the same position by Lemma~\ref{lm:bigear}.
The argument is a direct simplification of that in (i), considering
$T_1 \setminus E$ and $T_2 \setminus E$.
\qed



\section{Geometric properties of tiles and strands} \label{sec:2nd-machinery}
 

\begin{defn}
Fix $n$. Then an increasing subset $Q=\{ q_1, q_2, ...,q_r \} $  of 
$\{ 1,2,...,n \}$ defines two partitions:
\[
I(Q)=\{ [q_1,\dots, q_2), [q_2,\dots, q_3),\dots, [q_r,\dots, q_1) \}
\]
\[
J(Q) = \{ (q_1,\dots, q_2], (q_2,\dots, q_3],\dots, (q_r,\dots, q_1] \}
\]
We denote the parts by 
$I_i(Q):=[q_i,\dots, q_{i+1})$ and 
$J_i(Q):=(q_i,\dots, q_{i+1}]$, for $i=1,\dots, r$. 
\end{defn}

Such partitions arise from tilings: Let $Q\in A_n$. Then the 
vertices of $Q$ partition the vertices of $P$ in two ways. 
Consider the edge $e=[q_i,q_{i+1}]$ of $Q$. In the complement 
to $e$, there are $q_{i+1}-q_i$ strands starting at vertices 
in $I_i(Q)$ and the same number of strands ending at the 
vertices in $J_i(Q)$. Among them, $q_{i+1}-q_i-1$ remain in the 
complement. For an example, see Figure~\ref{fig:tile-strand}.

Using this notation, we get an alternative proof for Corollary~\ref{cor:only-one-strand} 
stating that a strand of a tiling can only use one strand segment of a given tile: 
Let $Q$ be a tile of a tiling $T\in A_n$
and let $q_1,\dots, q_r$ be its vertices, 
$r\ge 3$, $q_1<q_2<\cdots q_r<q_1$.  
Assume strand 
$x\leadsto y$ involves a strand segment of $Q$, say parallel to the 
edge $[q_{i-1},q_i]$.
By construction, this strand segment 
is oriented from $q_i$ to $q_{i-1}$. 
We claim that the strand then 
necessarily starts in $I_i(Q)$ and ends in $J_{i-2}(Q)$. 
Consider the edge $[q_{i-2},q_{i-1}]$ of $Q$: the only place for a strand to leave $J_{i-2}(Q)$ 
is near the vertex $q_{i-2}$. By the orientation of strand segments in tiles, 
strand $x\leadsto y$ could only leave near $q_{i-2}$ if $q_{i-2}=q_{i-1}$ --- a contradiction. 
Hence $y\in J_{i-2}(Q)$. 
A similar argument shows $x\in I_i(Q)$.

\begin{remark}\label{rm:long-vs-short}
Let $T$ be a tiling of $P$, with tile $Q$ inducing partitions as above. There 
are two types of strands regarding these partitions. 
Let $x\leadsto y$ be a strand starting in $I_{i_1}(Q)$ and ending in $J_{i_2}(Q)$ 
for some $i_1,i_2$. Then we either have $i_1=i_2$ or $i_1=i_2+2$ 
(by the preceding argument or by Corollary~\ref{cor:only-one-strand}). 
The case 
$i_1=i_2+2$ 
is illustrated in Figure~\ref{fig:tile-strand} for $Q$ a pentagon. 
\end{remark}

\begin{figure}[h]
\includegraphics[scale=.65]{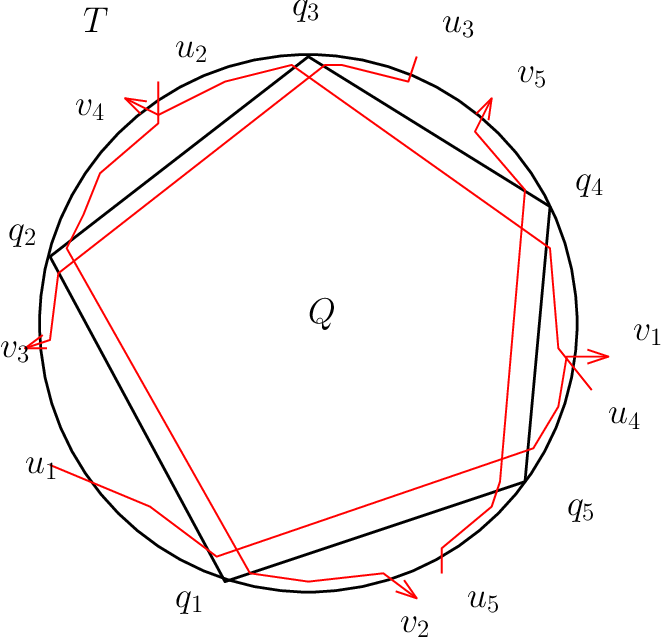} 
\caption{Tile inducing partition and long strands for $Q$}\label{fig:tile-strand}
\end{figure}

\begin{defn} \label{de:ls}
Let $Q$ be a tile of a tiling $T$ of an $n$-gon. If a strand 
$x\leadsto y$ of $T$ uses 
a strand segment of $Q$, we say that the strand $x\leadsto y$ is a 
{\em long strand for $Q$}. 
If for the partitions induced by $Q$, $x\in I_i$, then $y\in J_{i-2}$ if $x\leadsto y$ 
is a long strand for $Q$ 
and $y\in J_i(Q)$ otherwise, 
cf. Remark~\ref{rm:long-vs-short}. 
\end{defn}

\begin{lm}\label{lm:covering}
Let $Q$ be an $r$-tile of a tiling of $P$ with vertices $q_1<\dots< q_r<q_1$ clockwise. 
Then every long strand 
$x\leadsto y$ with respect to $Q$ covers
exactly $r-2$ vertices of $Q$ 
and there are exactly 
two vertices $q_{i-1},q_i$ for every such strand with $y\le q_{i-1}<q_i\le x<y$ 
(clockwise). 
\end{lm}

\begin{proof}
If $s$ is a long strand for $Q$ with $x\leadsto y$, then there exists $i$ such that 
$x\in I_i(Q)=[q_i,\dots, q_{i+1})$ and $y\in J_{i-2}(Q)=(q_{i-2},\dots, q_{i-1}]$ 
(reducing the index mod $n$), hence it covers $q_{i+1},q_{i+2}, \dots, q_{i-2}$. 
For an illustration, see Figure~\ref{fig:tile-strand}. 
\end{proof}


\section{Geometric Proof of Theorem}
   \label{sec:subdiv-to-permut}


We now use geometric properties of tilings to prove the ``only if'' 
part of 
Theorem~\ref{thm:tilings-permut-inj}.
The maximum tile size of tiling $T$ is denoted $r(T)$.
For two tilings $T_1,T_2$ and $r_i = r(T_i)$,
the case 
$r_1\ne r_2$ is covered in Lemma~\ref{lm:maximum-agrees};
and
$r:=r_1=r_2$
follows from
Corollary~\ref{cor:size5} and Lemma~\ref{lm:size4}. 
%
%
We first prove an auxiliary result. 

\begin{lm}\label{lm:hilfslemma}
(a) Consider a tiling $T$ in $A_n$ with a diagonal  $e=[s_1,s_2]$.
For each vertex $q$ with  $s_1 < q < s_2<s_1$,
we get a strand  $s:y\leadsto z$ in $\Scott(T)$ covering $q$, with 
$s_1\le y<q<z\le s_2<s_1$.
\\
(b) Consider $T,q,s$ as in (a) and a further tiling $T'$ of $P$ 
containing a tile $Q$ such that $\dim(Q \cap e)=1$ and $q \in Q$. 
If  $\Scott(T')$ contains a strand with $y\leadsto z$ as in (a), 
it is a long strand for $Q$ (as defined in Definition~\ref{de:ls}). 
\end{lm}

\begin{proof}
(a)
Let $e'=[n_1,n_2]$ be the shortest diagonal 
in $T$ lying above $q$. 
Note, $s_1\le n_1<q<n_2\le s_2<s_1$. 
Consider the strand segment in $\Scott(T)$
following $e'$ from $n_1$ to $n_2$ 
(see figure below).
This induces a strand $s$ with $y\leadsto z$, say. 
We claim $n_1\le y<q$ and $q<z\le n_2$. To see this let 
$[x,n_1]$ be in $T$ with $n_1\le x$ and $x\le q$ maximal ($x=n_1+1$ possibly). 
Then $s$
has its starting point among $\{n_1,n_1+1,\dots, x-1\}$,
because $e'$ is the shortest diagonal above $q$. 
Let $[y,n_2]$ be in $T$  with $y\le n_2$ and $y\ge q$ minimal. 
Then $s$ has its endpoint among $\{y+1,\dots, n_2\}$, similarly. 
$$
\includegraphics[scale=.5]{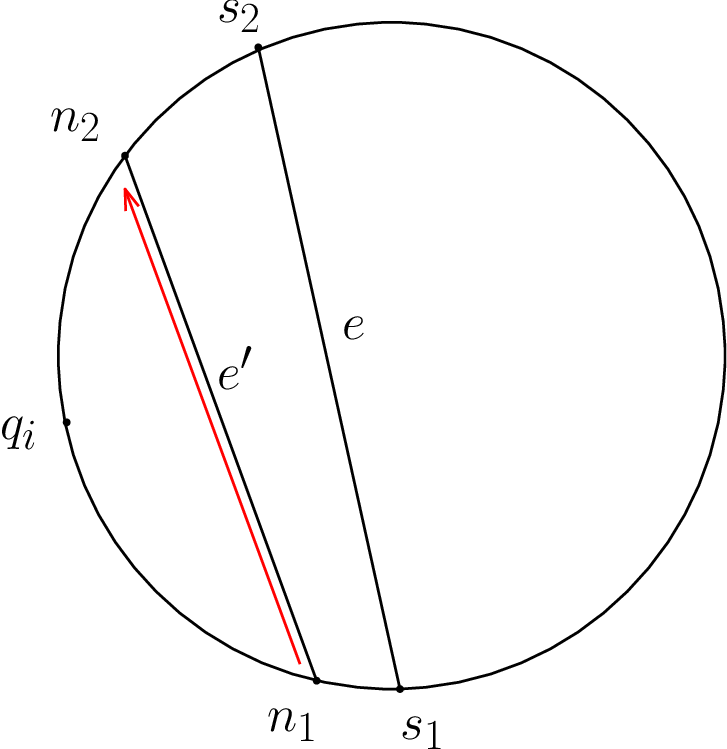}
$$
(b) Given (a), this follows immediately from Definition~\ref{de:ls}. 
\end{proof}



\begin{lm}\label{lm:maximum-agrees}
Let $T_1$ and $T_2$ be two tilings of a polygon $P$ with 
$\Scott(T_1) = \Scott(T_2)$.
Then $r_1 = r_2$. 
\end{lm}

\begin{proof}
Let $r=r_2$. 
In case $r_1=3$, the claim follows from
Lemma~\ref{lm:triangle-case}: in this case, 
$\Scott(T_1)$ is induced by $i\mapsto i+2$ and $T_2$ has to be 
a triangulation, too.  
Assume that  $\Scott(T_1)=\Scott(T_2)$ and for contradiction, 
assume that $r_1<r$. (Remark: By the above, we can assume $r>4$.) 
We consider a tile $Q$ of size $r$ in $T_2$, with vertices $q_1,\dots, q_r$. 
In $T_1$, we choose a tile $S$ 
with  $\dim Q\cap S> 1$. 
We write $\{ s_1, ...,s_s \}$ for the vertices of $S$, with  $|S|=s<r$. 

By this construction there is an edge $e=[s_1,s_2]$ of $S$ and $q_i \in Q$
with $s_1<q_i<s_2<s_1$.  
There is, therefore,  a strand with $y\leadsto z$ 
in $\Scott(T_1)$ covering $q_i$, as in Lemma~\ref{lm:hilfslemma}. 

Since  $\Scott(T_1)=\Scott(T_2)$,
under the tiling $T_2$ there exists a strand with $y\leadsto z$ covering $q_i$. 
By Lemma~\ref{lm:hilfslemma}(b) it is a long strand for $Q$. 
Consider the subpolygon $P'$ on the vertices $y,y+1,\dots, z$,
i.e. with interior edge $e'=[y,z]$.  
We have shown in Lemma~\ref{lm:covering} that $r-2$ vertices of $Q$ lie in $P'$, w.l.o.g. these 
are $\{q_1,\dots, q_{r-2}\}$. 
Furthermore these $r-2$ vertices are 
different from $y$, $z$ (again Lemma~\ref{lm:covering}). 
Thus, $P'$ has at least $r$ vertices. 

We consider the diagonals in $T_1$ with endpoints among 
$y,y+1,\dots,z$, covering at least one of $\{q_1,\dots, q_{r-2}\}$. 
Not all of these diagonals cover all $r-2$ vertices 
(otherwise we would get a tile of size $r$ in $T_1$,
specifically $\{ a,q_1,\dots ,q_{r-2},b \}$ where $[a,b]$ 
is the shortest covering diagonal). 
So let  $e''=[n_1,n_2]$ with $y\le n_1<n_2\le z<y$ 
cover $q_{l_1}$, say, but not $q_{l_2}$. 
By the same argument as above, this implies that there exists 
a strand with $y_2\leadsto z_2$ in $\Scott(T_1)$, 
with $n_1\le y_2 <q_{l_1}<z_2\le n_2<n_1$. 
We are in the situation here:
$$
\includegraphics[scale=.5]{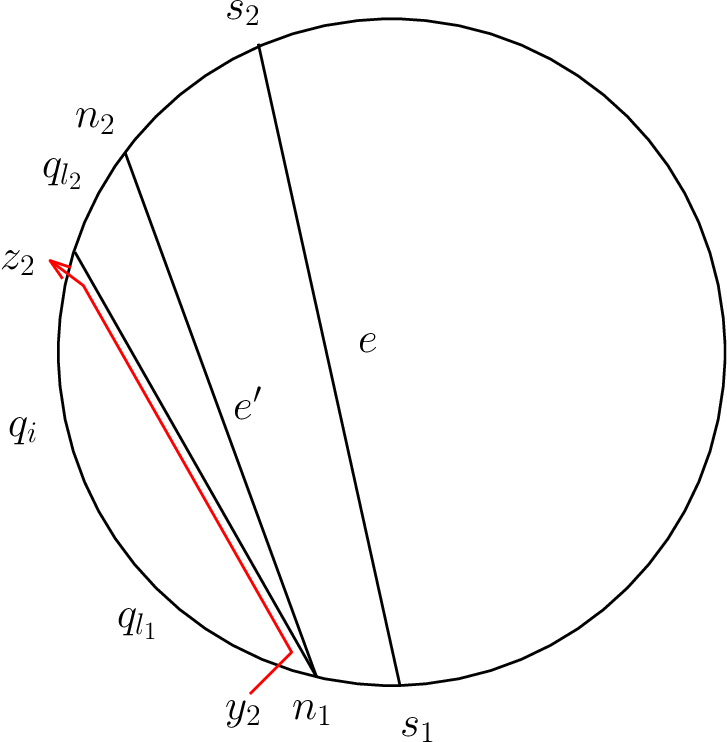}
$$
In $\Scott(T_2)$, 
the assumed strand from $y_2$ to $z_2$ is a long strand for $Q$
again by Lemma~\ref{lm:hilfslemma}. 
As before, this implies that among the vertices 
$n_1+1,n_1+2,\dots,n_2-1$, there are $r-2$ vertices of $Q$. But we already 
have $q_{r-1},q_r$ and $q_{l_2}$ lying outside. 
A contradiction. 
\end{proof}

\begin{cor} \label{cor:size5}
Let $\Scott(T_1)=\Scott(T_2)$. If in $T_1$ there exists an edge 
$e=[s_1,s_2]$ and in $Q$ a tile of size $\ge 4$ 
with $\dim Q\cap e=1$, 
then either $e$ is an edge of $Q$ 
or $e$ separates vertices of $Q$ 
($s_1 > q_x > s_2 > q_y > s_1$ for some $x,y$)
and $|Q|=4$.
\end{cor}
\begin{proof}
If $e$ is not an edge of $Q$, we find vertices $q_i$ and $q_j$ of $Q$ with 
$s_1<q_i<s_2<q_j<s_1$. We can thus use 
Lemma~\ref{lm:hilfslemma} and Lemma~\ref{lm:covering} for $q_i$ and 
again for $q_j$ to see that $Q$ has $r-2$ vertices on the left side of $e$ and 
$r-2$ vertices to the right of $e$, and that they all differ from $s_1$ and from $s_2$. 
\end{proof}


\begin{lm} \label{lm:size4}
Let $T_1$ and $T_2$ be two tilings of a polygon $P$ with 
$\Scott(T_1) = \Scott(T_2)$ and assume $r_1=r_2=4$.
Then $\te{T_1 } = \te{ T_2 }$. 
\end{lm}
\begin{proof}
By Lemma~\ref{lm:simple-edge-rotation} 
the positions of 4-ears in $T_1$ and $T_2$ agree, when $r_1=r_2=4$.
By the Factorisation Lemma (Lemma~\ref{lm:glue-tiles}) 
we can remove (common) ears of size 4, to leave reduced tilings
$T_1'$ and $T_2'$ of some $P'$.
These necessarily have ears, but by Lemma~\ref{lm:strand-tile-magic},
(up to equivalence) 3-ears 
can be chosen to be in the same positions in each tiling.
Now iterate. 
\end{proof}


\begin{proof}[Proof of Theorem~\ref{thm:tilings-permut-inj}] 
If the maximum tile sizes of $T_1$ and of $T_2$ differ, the claim follows from 
Lemma~\ref{lm:maximum-agrees}. So let $r=r_1=r_2$ be the maximum tile size 
of $T_1$ and of $T_2$. If $r=4$, Lemma~\ref{lm:size4} proves the claim. So assume 
that there are tiles of size $r>4$ and consider such a tile $Q$ in $T_2$. 
By Corollary~\ref{cor:size5} there are no diagonals of $T_1$ 
`intersecting' $Q$, so in $T_1$ we have a tile containing $Q$.
Applying the same argument with the tilings reversed we see that $T_1$
and $T_2$ agree on parts tiled with tiles of size $>4$.

By the Factorisation Lemma (Lemma~\ref{lm:glue-tiles}), we can remove all 
(common) ears of size $>4$. 
Among the remaining (common) tiles of size at least 5, we choose a tile $Q$ and 
a non-boundary edge $e$ of $Q$, such that to one side of $e$, 
all tiles in $T_1$ and in $T_2$ have size at most four. Let $P'$ be the union of these 
tiles of size $\le 4$. 
By the Factorisation Lemma we have $\Scott(T_1\mid_{P'})=\Scott(T_2\mid_{P'})$ and 
by Lemma~\ref{lm:size4}, $[T_1\mid_{P'}]=[T_2\mid_{P'}]$. 
We can remove $P'$ and repeat the above until $Q$ is a (common) ear - which 
can be removed, too. 
Iterating this proves the claim. 
\end{proof}



\newcommand{\figd}{../fig/}
\newcommand{\geometricmutation}{flip}  


\section{On the image of the Scott map treated combinatorially}\label{sec:count}

To give an intrinsic characterization of the image in
$\Sym_n$ of the Scott map
$\Scott: A_n \rightarrow \Sym_n$ for all $n$ 
remains an interesting open problem.  
Note in particular  
that so far the map does not equip the image
with a group structure (or indeed any algebraic structure). 
Here 
we report on  one invariant which 
Theorem~\ref{thm:tilings-permut-inj} gives us access to, namely
the size of the image,
which is given by $|\AE_n |$.  


As an initial illustration we observe that: 
\begin{prop} \label{pr:ncr}
The 
number of permutations 
arising from tiling an $n$-gon using one $r$-gon ($r>3$) and triangles
otherwise  is $n\choose r$.
\end{prop}
\proof{ By the main Theorem this is the same as enumerating the classes
  in $\AE_n$ of this type.
Since the details of the triangulated part are irrelevant,  
the class is determined by choosing the vertices of the $r$-gon. 
Hence choosing $r$ from $n$. \qed}

\begin{ex}
In total, there are 26 permutations arising from 
the 45 tilings of the hexagon: 
one from the empty tiling; 
$6$ from tilings with one pentagon and one triangle; 
$15$ from tilings with one quadrilateral and two triangles; 
$3$ from tilings using two quadrilaterals; and 
1 from the triangulation case. \\
Figure~\ref{fig:hexagon} contains examples of these tilings and the associated 
permutations. 
\end{ex}

In order to go further we will need some notation. 


\subsection{Notation and known results}\label{sec:cayleyans}
\newcommand{\ralph}{r^{\alpha_r} (r-1)^{\alpha_{r-1}} \cdots 2^{\alpha_2}
  1^{\alpha_1}}

Recall that an 
integer partition $\lambda=(\lambda_1,\lambda_2,\dots)$ 
has also the exponent notation:
$$
\lambda = r^{\alpha_r} (r-1)^{\alpha_{r-1}} \cdots 2^{\alpha_2} 1^{\alpha_1}
$$
A {\em $\lambda$-tiling} is a tiling with, for each $d$,
$\alpha_d$ tiles that are $(d+2)$-gonal. 

Recall that 
$A_n$ is the complex 
of tilings of the $n$-gon. Define $a_n = |A_n| $.
Write
$A_n(m)$ (with $m \in \{0,1,2,\dots, n-3 \}$)
for the set (and $a_n(m)$ the number) of tilings with $m$
diagonals. 
Write
$A_n(\lambda)$ (with $\lambda$ an integer partition of $n-2$) 
for the set of $\lambda$-tilings 
(thus with a $m$-gonal face for each row $\lambda_i = m-2$). 
Thus
\begin{equation} \label{eq:mlam}
A_n(m) \; = \bigcup_{\lambda \vdash n-2 \; : \; \lambda'_1 = m+1 } A_n(\lambda)
\end{equation}



Similarly recall
$\AE_n$ is the set of classes of tilings under triangulated-part/flip
equivalence. Write 
$\AE_n(m)$ for the set $A_n(m)$ under triangulated-part equivalence and 
$\AE_n(\lambda)$ the set $A_n(\lambda)$ under triangulated-part equivalence. 


\mdef
The sequence $a_n$ is   
the little Schr\"oder numbers (see e.g. \cite{Stanley97} and 
OEIS A001003).
It is related to the Fuss--Euler combinatoric as follows. 
By \cite{ps} the number of tilings of the $n$-gon with $m$
diagonals is 
\begin{equation}  \label{eq:qm1n}
a_n(m) = \frac{1}{m+1}{n+m-1\choose m}{n-3\choose m} 
\end{equation}
($a_n(m) = q_m(1,n) $ from \cite{ps}). 
Thus
in addition to the usual generating function
\[
\sum_{n \geq 0} a_n x^n = \frac{1+x-\sqrt{1-6x+x^2}}{4x}
\]
we have
\begin{equation}  \label{eq:a-n}
a_n = \sum_{m=0}^n\frac{1}{m+1}{n+m-1\choose m}{n-3\choose m}
\end{equation}

\subsection{Explicit construction of $A_n$}
Of greater use than an expression for the size of $A_n$ is an
explicit construction of all tilings. 
For this we shall consider a tiling in $A_n$ to be 
as in the formal definition, i.e. to be the same as its set of arcs. 
This is the set of diagonals in the present polygon case, 
where we can
represent an arc between vertices $i,j$ 
unambiguously by $\di{i,j}$. 
In particular then we have an inclusion 
$A_{n-1} \hookrightarrow A_n$. 
The copy of $A_{n-1}$ in $A_n$ is precisely the subset of tilings in
which vertex $n$ is simple and there is no diagonal $\di{ 1,n-1 }$. 

There is a disjoint image $J(A_{n-1})$ of $A_{n-1}$ in $A_n$ given by 
$J(T)=T \cup \{\di{ 1, n-1 }\}$. 
The set $A_{n-1} \sqcup J(A_{n-1})$ is the subset of $A_n$ of
elements in which $n$ is simple. 
Consider in the complement the subset of tilings containing 
$\di{n-2,n }$. 
In this the vertex $n-1$ is necessarily simple. 
Thus this subset
is the analogue  $J_{n-1}(A_{n-1})$ of $J(A_{n-1})$ constructed with
$n-1$ instead of $n$ as the distinguished simple vertex. The practical
difference is that (i) the image tilings have all occurences of $n-1$
replaced by $n$; (ii) the `added' diagonal is $\di{n-2,n }$.


There remain in $A_n$ the tilings in which $n$ is not simple but there
is not a diagonal $\di{n-2,n }$. Consider those for which there is a
diagonal $\di{ n-3, n }$. 
In the presence of this diagonal any tiling `factorises' into the
parts in the two subpolygons on either side of this diagonal. 
One of these has vertices $1,2,...,n-3$ and $n$, and so its tilings
are an image of $A_{n-2}$ where vertex $n-2$ becomes vertex $n$.
The other has vertices $n-3,n-2,n-1,n$ 
and so has tilings from a
shifted image of $A_4$, but has $n$ simple (since $\di{ n-3 , n }$ is
the first diagonal in the original tiling).
Since $n$ is simple, it 
 is the part of that image coming from $A_3 \sqcup J(A_3)$.
We write $2.K(A_3)$ for these two shifted copied of $A_3$.
We write $2.K(A_3)\cdot A_{n-2}$ for the meld with tilings from $A_{n-2}$
to construct the set of tilings of the original polygon.

There now remain in $A_n$ the tilings in which $n$ is not simple but there
is not a diagonal $\di{n-2,n }$ or $\di{ n-3,n }$. 
Consider those for which there is a diagonal $\di{ n-4, n }$. 
In the presence of this diagonal any tiling `factorises' into the
parts in the two subpolygons on either side of this diagonal. 
We have the obvious generalisation of the preceeding construction in this case,
written $2.K(A_4)\cdot A_{n-3}$.

We may iterate this construction until all cases of diagonals from $n$
are included. We have established the following.

\newcommand{\lA}{{\mathbf A}} 
\newcommand{\cupp}{\cup}  

\begin{prop}
Consider the list defined recursively by $\lA_3 = (\emptyset )$ and 
$$
\lA_n \; = \; \lA_{n-1} \cupp J(\lA_{n-1}) \cupp J_{n-1}(\lA_{n-1}) 
                \cupp \bigcup_{r=2}^{n-3} 2.K(\lA_{r+1})\cdot\lA_{n-r} 
$$
where set operations on lists are considered as concatenation in the
natural order; $2.-$ denotes the doubling as above;
$K()$ denotes the relabeling of all vertices so that
the argument describes a suitable subpolygon;
and $A\cdot B$ denotes the meld of tilings from subpolygons as above.
Then this list is precisely a total order of $A_n$. 
\end{prop}
\begin{proof}
Noting the argument preceding the Proposition, 
it remains to lift the construction from the set to the list.
But this requires only the interpretation of union as concatenation.
\end{proof}

\subsection{Tables for $A_n(\lambda)$}\label{sec:cayleyanswer}


The class sets $\AE_n$ are harder to enumerate than $A_n$. 
Practically, one approach is to list elements of
$A_n$ and organise by arrangement of their triangulated parts, which
determines the class size.
We first recall the numbers $a_n(m)$ of tilings of an $n$-gon with $m$
diagonals: see Table~\ref{tab:1}.
\begin{table} 
\[
\begin{array}{l|l|c|l|l|l|l|l|l|l|}
n & \Sigma & m=0  & 1  & 2    & 3   & 4    & 5 &6 & 7\\ \hline
3 & 1      & 1  &&&&&&&\\
4 & 3      & 1    & 2  & &&&&&\\ 
5 & 11     & 1    & 5  & 5&&&&& \\
6 & 45     & 1    & 9  & 21   & 14 &&&&\\
7 & 197    & 1    & 14 & 56   & 84  & 42 &&&\\
8 & 903    & 1    & 20 & 120  & 300 & 330  & 132&& \\
9 & 4279   & 1    & 27 & 225  & 825 & 1485 & 1287 &  429 &\\
10 & 20793 & 1    & 35 & 385 & 1925 & 5005 & 7007 & 5005 & 1430 \\
\hline
n &        & 1    & {\tiny\frac{n(n-3)}{2}} & 
          \!\!\!{\tiny\frac{{\! n+1\choose 2}\!{\! n-3\choose 2\!}}{3}}\!\!
                          &     & && &
\end{array}
\]
\caption{ Values of $a_n(m)$, and hence $a_n$, in low rank. \label{tab:1}}
\end{table}
%
The main diagonal enumerates the top dimensional simplices in $A_n$. 
It counts triangulations and hence is the Catalan sequence $C_n$. 
The entries in the next diagonal 
\ignore{
are the simplices with one 
are the simplices just below the top-dimensional ones \footnote{check language!}, they} 
correspond to tilings with a single quadrilateral and triangles else. 

\begin{table}
\[
\begin{array}{l|l|c|l|l|l|l|l|l|l|}
n & \Sigma & m=0  & 1  & 2   & 3   & 4  & 5 &6 &7  \\ \hline
3 & 1      & 1  &&&&&&&\\
4 & 2      & 1    & 1  &&&&&& \\ 
5 & 7      & 1    & 5  & 1 &&&&&\\
6 & 26     & 1    & 9  & 15  & 1 &&&&\\
7 & 100    & 1    & 14 & 49  & 35  & 1 &&&\\
8 & 404    & 1    & 20 & 112 & 200 & 70 & 1& &\\
9 &   1691 & 1    & 27 & 216 &  654 &  666  & 126 & 1& \\
10 &   7254 & 1    & 35 &  375  & 1660 & 
3070 & 1902  & 210 & 1 
\end{array}
\]
\caption{Table of $\AE_n$ sizes up to $n=10$. \label{tab:2}}
\end{table}
%

We will give the number of elements of $\AE_n(m)$ for
small $n$ in Table~\ref{tab:2}. 
In order to verify this it will be convenient to 
refine  tables \ref{tab:1} and \ref{tab:2} 
by considering these numbers for fixed partitions $\lambda$. 
Specifically we subdivide each case of $m$ from the previous tables
according to $\lambda$, with the $m$-th composite entry written 
as a list of entries in the form 
$(\lambda_1, \lambda_2, ...) \over a_n(\lambda)$
ranging over all $\lambda$ with $|\lambda | =m$. 
Thus for example $(32) \over 7$ tells that $a_7((3,2))=7$. 
We include  
Table~\ref{tab:Anl} for $A_n(\lambda)$ and  
Table~\ref{fig:AE-lambda} for $\AE_n(\lambda)$. 
Neither table is known previously. 
The $a_n$ case is computed partly by brute force (and see below);
verified in GAP \cite{GAP}, and checked 
using identity (\ref{eq:mlam})). 

For the purpose of computing $\AE_n$ a better filtration is by the
partition describing the size of the connected triangulated
regions. But this is even harder to compute in general.


\newcommand{\st}[2]{\begin{array}{ccccc} (#1) \\ \hline #2 \end{array}}
\newcommand{\sto}[2]{\begin{array}{ccccc} #1 \\ \hline #2 \end{array}}

\begin{table}
\[
\hspace{-.5in} {
{\rotatebox{90}{
\scalebox{.74}{
$
\begin{array}{c}
\begin{array}{l|l|cllllllllllll}
n & \Sigma &\; |\lambda|= {0}  & 1  & 2   & 3   & 4  & 5 & 6 & 7  \\ \hline
3 & 1      & \st{1}{1}  \\
4 & 3      & \st{2}{1}    & \st{1^2}{2}  & \\ 
5 & 11     & \st{3}{1}    & \st{21}{5}      & \st{1^3}{5} \\
6 & 45     & \st{4}{1}    & \st{31)&(2^2}{6&3} &\st{21^2}{21} & \st{1^4}{14} \\
7 & 197    & \st{5}{1}    & \st{41)&(32}{7&7} & \st{31^2)&(2^21}{28&28}  & \st{21^3}{84}  & \st{1^5}{42} \\
8 & 903    & \st{6}{1}    & \st{51)&(42)&(3^2}{8&8&4} &
  \st{41^2)&(321)&(2^3}{36&72&12} & \st{31^3)&(2^2 1^2}{120&180} & 
    \st{21^4}{330} & \st{1^6}{132} \\
9 & 4279   & \st{7}{1}    & \st{61)&(52)&(43}{9&9&9} & 
  \st{51^2)&(421)&(3^2 1)&(32^2}{45&90&45&45} & \st{41^3)&(321^2)&(2^3 1}{165&495&165} &
    \st{3 1^4)&(2^2 1^3}{495&990} & \st{21^5}{1287} & \st{1^7}{429}\\
10 &   20793  & \st{8}{1}    & \st{71)&(62)&(53)&(4^2}{10&10&10&5} & 
  \st{61^2)&(521)&(431)&(42^2)&(3^22}{55&110&110&55&55} & %
  \st{51^3)&(421^2)&(3^21^2)&(32^21)&(2^4}{220& 660 & 330 &660& 55} &
    \st{4 1^4)&(32 1^3)&(2^31^2}{  715&2860 &1430} & \st{31^5)&(2^21^4}{2002& 5005 } & \st{21^6}{5005} & \st{1^8}{1430}
\end{array}
\\ \\ \mbox{Table~\ref{tab:Anl}: Table of $A_n(\lambda)$.}
\end{array}
$
}  
}}
}
\]
\caption{Table for size of $A_n(\lambda)$. \label{tab:Anl}}
\end{table}


\begin{table}
\[ 
\hspace{-.5in} {
{\rotatebox{90}{
\scalebox{.75}{%
$
\begin{array}{c} 
\begin{array}{l|l|cllllllllll}
n & \Sigma &\; |\lambda|= {0}  & 1  & 2   & 3   & 4  & 5  & 6 & 7\\ \hline
3 & 1      & \st{1}{1}  \\
4 & 2      & \st{2}{1}    & \st{1^2}{1}  & \\ 
5 & 7      & \st{3}{1}    & \st{21}{5}      & \st{1^3}{1} \\
6 & 26     & \st{4}{1}    & \st{31)&(2^2}{6&3} &\st{21^2}{15} & \st{1^4}{1} \\
7 & 100    & \st{5}{1}    & \st{41)&(32}{7&7} & \st{31^2)&(2^21}{21&28}  & \st{21^3}{35}  & \st{1^5}{1} \\
8 & 404    & \st{6}{1}    & \st{51)&(42)&(3^2}{8&8&4} &
  \st{41^2)&(321)&(2^3}{28&72&12} & \st{31^3)&(2^2 1^2}{56&144} & 
    \st{21^4}{70} & \st{1^6}{1} \\
9 &  1691     & \st{7}{1}    & \st{61)&(52)&(43}{9&9&9} & 
  \st{51^2)&(421)&(3^2 1)&(32^2}{ 36& 90 &45 &45 } & \st{41^3)&(321^2)&(2^3 1}{ 84& 405 & 165 } &
    \st{3 1^4)&(2^2 1^3}{ 126 &  540  } & \st{21^5}{126 } & \st{1^7}{1}\\
10 &    7254   & \st{8}{1}    & \st{71)&(62)&(53)&(4^2}{10&10&10&5} & 
  \st{61^2)&(521)&(431)&(42^2)&(3^22}{45&110&110&55&55} & %
 \st{51^3)&(421^2)&(3^21^2)&(32^21)&(2^4}{120& 550& 275&660& 55} & 
   \st{4 1^4)&(321^3)&(2^31^2}{210& 1650&1210} & \st{31^5)&(2^21^4}{252 &  1650} & \st{21^6}{210} & \st{1^8}{1}
\end{array} 
\\ \\
\mbox{Table~\ref{fig:AE-lambda}: 
      Table of $\AE_n(\lambda)$.} 
\end{array}
$ 
}}}
}
\] 
\caption{Table for size of \AE$_n(\lambda)$}\label{fig:AE-lambda}
\end{table}

\ignore{
\kb{in the table for \AE\ I have uncolored the entries for $n=9$  
that seem clear (the ones that 
are the same in $A_9$ and the ones with two triangles only. And the ones with 
one $r$-gon and s triangles else, as they are r choose s).}
}


%

\subsection{Formulae  
  for $|\AE_n(\lambda)|$ for all $n$.} $\quad$ \label{sec:AEn}

In the $\lambda$ notation Proposition~\ref{pr:ncr} becomes
\begin{equation} \label{eq:ncr2}
| \AE_n((r-2)1^{n-r}) | \; = \; {n \choose r}
\end{equation}

To determine the size of image of the Scott map for a polygon of a
given rank, one  
strategy is to compute $\AE_n(\lambda)$ through $A_n(\lambda)$. 
While $A_n(\lambda)$ is also not known in general,  we have a GAP code 
\cite{GAP,BMGAP} to compute any given case. 

If in a tiling, there is at most one triangle, we have 
$\AE_n(\lambda) \cong A_n(\lambda)$. 
In the case of two triangles, the following result  
determines $|\AE_n(\lambda)|$ 
from tilings of the same type and from tilings where 
the two triangles are replaced by a quadrilateral: 

\newcommand{\aA}{a_n} 

\begin{prop}\label{prop:reduction}
Let $\lambda=\ralph$. \\
(i) If $\alpha_1 < 2$ then $|\AE_n(\lambda)| = a_n(\lambda)$. \\
(ii) If $\alpha_1=2$ then 
\[
| \AE_n(\lambda) | \; = \; a_n(\lambda) - (\alpha_2 +1) a_n(\lambda 21^{-2})
\]
(iii)  If $\alpha_1=3$ then 
\[
| \AE_n(\lambda) | \;
     = \aA(\lambda) - (\alpha_2 +1) \aA(\lambda 21^{-2}) 
                  + (\alpha_3 +1) \aA(\lambda31^{-3})
\]
(iv)  If $\alpha_1=4$ then 
\[ 
| \AE_n(\lambda) | \;
   = \aA(\lambda) - (\alpha_{4}+1) \aA(\lambda 41^{-4})
   +(\alpha_3 +1) \aA(\lambda 31^{-3})
   \hspace{1in} \]
   \[ \hspace{2.1in}
             +{\alpha_2 +2 \choose 2} \aA(\lambda 221^{-4})
             -(\alpha_2 +1) \aA(\lambda 21^{-2}) 
\]
\end{prop}
\proof{ (ii) Consider partitioning $A=A_n(\lambda)$ into a subset $A'$ of 
tilings where the triangles are adjacent, and $A''$ where they are not.
Evidently $ | \AE_n(\lambda) | \; = \; |A'|/2 + |A''| = |A|-|A'|/2$. 
On the other hand in $A'$ the triangles form a distinguished
quadrilateral. For each element of $A_n(\lambda21^{-2})$ we get
$\alpha_2 +1$ ways of selecting a distinguished quadrilateral.
There are two ways of subdividing this quadrilateral, 
thus $|A'| = 2(\alpha_2 +1) a_n(\lambda 21^{-2})$. 
\qed }

%


\begin{ex}
Proposition~\ref{prop:reduction} determines $|\AE_8(2^2 1^2)|$. 
Here $A_8(2^2 1^2)$ gives an
overcount because of the elements where the two triangles are
adjacent.
Only one representative of each pair under \geometricmutation\ should
be kept. 
These are counted by marking one quadrilateral in each
element of $A_8(2^3)$. There are three ways of doing this, so we have
\[
|\AE_8(2^2 1^2)| = |A_8(2^2 1^2)| - 3 |A_8(2^3) |   = 180-36
\]
from Table~\ref{tab:Anl}. Similarly
$
|\AE_8(4 1^2)| = |A_8(4 1^2)| -  |A_8(42) |   = 36-8 .
$
\end{ex}

%

\mdef
{\em Proof of (iii):}
For  $\alpha_1 = 3$  
partition $A=A_n(\lambda)$ into subset $A'$ of tilings with three
triangles together; $A''$ with two together; and $A'''$ with all
separate. 
We have $|\AE_n(\lambda)| = |A'''| + |A''|/2 + |A'|/5$.
That is, 
\begin{equation}
|\AE_n(\lambda)| = |A| - |A''|/2 -4 |A'|/5 .
\end{equation}
%
%
Considering the triangulated pentagon in a tiling $T$ in $A'$ as a
distinguished pentagon we have 
\begin{equation} \label{eq:Aprime}
|A'| = 5(\alpha_3 +1) a_n(\lambda 31^{-3}) .
\end{equation}
\ignore{ ...
(2) The latter overcount is given by 
$4(\alpha_3 +1))A_n(\lambda 31^{-3})$. 
The 4 is $C_3 -1$, arising 
since we should keep one representative from $C_3 = 5$
triangulations of the given 5-gon. 
The $(\alpha_3 +1)$ is the number of ways of
choosing a 5-gon to mark in an element of $A_n(\lambda 3 1^{-3})$.
\\
(1) ... But the former is more involved... We
}
Next aiming to enumerate $A''$, consider $\lambda 21^{-2}$,
somewhat as in the proof of (ii), 
but here there is another triangle, which must not touch the marked 4-gon. 
Let us write $(\alpha_2 +1) A(\lambda 21^{-2})$ to denote a version of 
$A(\lambda 21^{-2})$ where one of the quads is marked.
There are two ways of triangulating the marked quad, giving 
$X=2(\alpha_2 +1) A(\lambda 21^{-2})$, say.
Consider the subset $B$ of $X$ of tilings where the marked quadrilateral and triangle are
not adjacent.

Claim: $B \cong A''$.
\\
Proof: The construction (forgetting the mark) defines a map $B \rightarrow A''$.
Marking the adjacent pair of triangles in an element of $A''$
gives a map $A'' \rightarrow B$ that is inverse to it. \qed

The complementary subset $C$ of $X$ has quadrilateral and triangle adjacent. 
Elements map into $A'$ by forgetting the mark.

Claim: $C$ double counts $A'$, i.e. the forget-map is surjective but
not injective. 
\\
Proof: There are 5 ways the quadrilateral and triangle can occupy a pentagon
together, and two ways of triangulating the quad. The cases can be
written out, and this double-counts the triangulations of the
pentagon. \qed


Altogether
$
A'' =  B = X-C = X-2A'
$ so 
\[
\AE(\lambda) = A(\lambda) - ((1/2)X-A') - (4/5) A'
     = A(\lambda) - (X/2) + A'/5 
\] \[ \hspace{1in}
     = A(\lambda) - (\alpha_2 +1) A(\lambda 21^{-2}) 
                  + (\alpha_3 +1) A(\lambda31^{-3})
\]
\qed





\mdef 
{\em Proof of (iv):}
For $\alpha_1 = 4$ 
partition $A=A(\lambda)$ by 
$
A = A^4 + A^{31} + A^{22} + A^{211} + A^{1111}
$
so that 
$$
\AE  = A^4/C_4 + A^{31}/C_3 + A^{22}/C_2^2 + A^{211}/C_2 + A^{1111} 
    = A - \frac{13}{14} A^4 - \frac{4}{5} A^{31} - \frac{3}{4} A^{22}
        -\frac{1}{2} A^{211}
$$

By direct analogy with (\ref{eq:Aprime}) we claim
$$
A^4 = 14(\alpha_{4}+1) A(\lambda 41^{-4})
$$
Next consider $X=5(\alpha_3 +1) A(\lambda 31^{-3})$, marking one
5-gon,
and then triangulating it.
We have a subset $B$ where the 5-gon and triangle are not adjacent;
and complement $C$.

Claim: $B \cong A^{31}$. This follows as in the proof of part (iii). 

The complement $C$ maps to $A^4$   
by forgetting the mark. 
 
Claim: 
$14 |C| = 30 |A^4|$. 
\\
Proof: There are 6 ways the 5-gon and triangle can occupy a hexagon
together, and 5 ways to triangulate the 5-gon. 
This gives 30 marked cases, which pass to 14 triangulations.

So far we have that 
$$
A^{31} = B = X-C = 5(\alpha_3 +1) A(\lambda 31^{-3}) -\frac{30}{14} A^4
$$
It remains to determine $A^{22}$ and $A^{211}$. 


\mdef
Next consider $Y= 4{(\alpha_2 +2) \choose 2} A(\lambda 221^{-4})$,
marking two 4-gons, and then triangulating them.
Subset $D$ has the 4-gons non-adjacent; and $E$ is the complement.

Claim: $D \cong A^{22}$. This follows similarly as the statement on $B$. 

The complement $E$ maps to $A^4$ by forgetting the marks. 

Claim: $14 |E| = 12 |A^4 |$.
\\
Proof: There are 3 ways the 4-gons can occupy a hexagon together,
and 4 ways to triangulate them. 
(NB the map is not surjective --- not every triangulation of a hexagon
resolves as two quadrilateral triangulations --- but we only need to
get the count right. We always get 12 out of 14 possible in each case.)

So far we have
$$
A^{22} = Y-E 
 = 4{{\alpha_2 +2} \choose 2} A(\lambda 221^{-4}) - \frac{12}{14}|A^4|
$$
Next we need $A^{211}$. 


\mdef
Next consider $Z=2 (\alpha_2 +1) A(\lambda 21^{-2})$, marking a 4-gon
and triangulating it. 
Subset $F$ has the three parts non-adjacent. Subset $G$ has the 4-gon
and one triangle adjacent. Subset $G'$ has the two triangles
adjacent. Subset $H$ has all three parts adjacent:
$$
Z = F + G + G' + H
$$

Claim: $F \cong A^{211}$. This follows similarly as the statements on $B$ and on $C$. 

The set $G$ maps to $A^{31}$, and $G'$ to $A^{22}$, and $H$ to $A^4$,
by forgetting the marks.

Claim: (a) $ |G| = 2 |A^{31} |$ and (b) $ |G'| =  2|A^{22} |$
  and (c) $14 |H| = 42 |A^4 |$.
\\ Proof: (a) Elements of $G$ pass to tilings with 
triangulations of a 5-gon and a separate triangle. 
The collection of them triangulating a given 5-gon
and triangle has order 10 (5 ways to mark a quadrilateral in the 5-gon, then
two ways to triangulate it). On the other hand the number of
triangulations of the same region in $A^{31}$ is 5.
\\ (b) Elements of $G'$ pass to tilings with triangulations of two 4-gons. 
The collection of such gives all these triangulations. Each one occurs
twice in $G'$ since the triangulation of the two 4-gon regions can arise in
$G'$ with one or the other starting out as the marked 4-gon. 
\\ (c) Elements of $H$ pass to tilings with triangulations of a
hexagon.
The collection of such gives $A_6(21^2) = 21$ ways of tiling the
hexagon with quadrilateral and two triangles, then two ways of tiling the quad.
On the other hand there are 14 triangulations of this hexagon in $A^4$.

We have $
A^{211} = Z-(G+G'+H) = 2 (\alpha_2 +1) A(\lambda 21^{-2})
                          - ( \frac{2}{1} A^{31} + \frac{2}{1} A^{22}
                                    +\frac{42}{14} A^4 ) .
$


Altogether now
$$
\AE(\lambda)  = A-\frac{13}{14} A^4 - \frac{4}{5} A^{31} - \frac{3}{4} A^{22}
                 -\frac{1}{2} A^{211}
\hspace{1in} $$ $$  
= A(\lambda) 
       -\frac{13}{14} A^4
       -\frac{4}{5} \left( 5(\alpha_3 +1) A(\lambda 31^{-3}) 
                              -\frac{30}{14} A^4 \right)
       -\frac{3}{4} \left( 4{\alpha_2 +2 \choose 2} A(\lambda 221^{-4}) 
                            - \frac{12}{14}|A^4| \right)
$$ $$       -\frac{1}{2} \left( 2 (\alpha_2 +1) A(\lambda 21^{-2})
                            - ( \frac{2}{1} A^{31} + \frac{2}{1} A^{22}
                                   +\frac{42}{14} A^4 )     \right)
$$ $$  
= A(\lambda) 
       +\frac{-13+21}{14} A^4
       +\frac{1}{5} \left( 5(\alpha_3 +1) A(\lambda 31^{-3}) 
                              -\frac{30}{14} A^4 \right)
       +\frac{1}{4} \left( 4{\alpha_2 +2 \choose 2} A(\lambda 221^{-4}) 
                            - \frac{12}{14}|A^4| \right)
$$ $$       -\frac{1}{2} \left( 2 (\alpha_2 +1) A(\lambda 21^{-2})
                                     \right)
$$ $$
= A(\lambda) + \frac{-13-6-3+21}{14} A^4 
             +(\alpha_3 +1) A(\lambda 31^{-3})
             +{\alpha_2 +2 \choose 2} A(\lambda 221^{-4})
             -(\alpha_2 +1) A(\lambda 21^{-2}) 
$$ $$
= A(\lambda) - (\alpha_{4}+1) A(\lambda 41^{-4})
             +(\alpha_3 +1) A(\lambda 31^{-3})
             +{\alpha_2 +2 \choose 2} A(\lambda 221^{-4})
             -(\alpha_2 +1) A(\lambda 21^{-2}) 
$$
\qed


\ignore{{

\begin{ex}
Checking the entries in 
\[
\AE(\lambda)  = A-\frac{13}{14} A^4 - \frac{4}{5} A^{31} - \frac{3}{4} A^{22}
                 -\frac{1}{2} A^{211}
\]

The formula for $A^{211}$ seems off. \\

(a) Take $\lambda=21^4$. The number of tilings where the 4 triangles form a hexagon is $14\times 8$. 
The number of tilings, where the triangles form a pentagon and a triangle (but not a hexagon) 
is 120 ($15\times 8$). The number of tilings where they form two quadrilateral, but not 
a hexagon: 48. And the number of tilings where they form a quadrilateral, two triangles, all 
disjoint, is 48. 

On the other hand, $A^4=14\times 8$, $A^{31}=120$, $A^{22}=48$, $A^{211}=96$. 

(b) Take $\lambda=31^4$. Here, the formula predicts 180. 
Counting tilings manually, $A^{21^2}$ is 108. With this, it works (correct number for \AE$_9(\lambda)$). 

(c) Take $\lambda=41^4$. The formula gives $A^{211}=300$ and I count 200. With the 200, 
I get the correct number for \AE$_9(\lambda)$. 

(d) Take $\lambda=2^21^4$. The formula gives $A^{211}=2310$. With the modified version 
(with $|G'|=2|A^{22}|$) I get $A^{211}=1650$. And then, miraculously, 
$|\AE(2^21^4)|=1650$ (as the manual count in SubdivisionsL). 
\end{ex}

}}

\subsection{Tables for $\AE_n$}\label{sec:countAEn}

\begin{prop}
The numbers $\AE_n$ for $n<11$ are given in Table~\ref{tab:2}.
\end{prop}
\proof{
The numbers $a_n(\lambda)$ are given in Table~\ref{tab:Anl} by a GAP
calculation \cite{BMGAP}. 
The numbers in Table~\ref{fig:AE-lambda} then follow from 
formula~(\ref{eq:ncr2}) 
and Proposition~\ref{prop:reduction}. 
Table~\ref{tab:2} follows immediately.
\qed
}

\subsection{On asymptotics}

We determined in tables \ref{tab:2}, \ref{fig:AE-lambda} the sizes 
of the image of the Scott map in low rank. 
Of course the ratio of successive sizes of the formal codomains 
grows with $n$ as  
$|\Sym_n| / |\Sym_{n-1} | = n$.
In the next table we consider the
ratios of two consecutive entries of the sequence $|\AE_n|_n$. 

\begin{tabular}{l|rrrrrrrr}
$n$ & 3 & 4 & 5 & 6 & 7 & 8 & 9 & 10 \\
\hline 
$|\AE_n|$ & 1 & 2 & 7 & 26 & 100 & 404 & 1691 & 7254 \\
$|\AE_n|/|\AE_{n-1}|$ &&  2 & 3.5 & 3.71 & 3.85 & 4.04 & 4.19 &  4.29 
\end{tabular}

\mdef
A paradigm for this is the Catalan combinatoric $C_n$
(see e.g. \cite{Stanley97}), 
which can also be equipped with an inclusion
in the permutations $\Sigma_n$ ---
see e.g. \cite{
Knuth98,TL321perm}
(NB this inclusion is not related to the inclusion in $A_n$ already noted).
It is straightforward in this case to verify that the asymptotic
growth rate is 4.

\begin{tabular}{l|rrrrrrrr}
$n$ & 3 & 4 & 5 & 6 & 7 & 8 & 9 & 10 \\
\hline 
$|C_n|$ & 1 & 2 & 5 & 14 & 42 & 132 & 429 & 1430 \\
$|C_n|/|C_{n-1}|$ &&  2 & 2.5 & 2.8 & 3   & 3.14 & 3.25  & 3.33 
\end{tabular}

\noindent
This raises the question:  
Is there a limit rate in the $\AE_n$ case?

\ignore{{
\begin{conj}
We conjecture that the ratio of successive entries in the size of the image of the Scott map 
is smaller than the ratio of successive sizes of the corresponding permutation groups. \\
\kub{With the maximum of the increase/slope around the entry 404?? (From $n=7$ to $n=8$, 
factor 4.04, from $n=8$ to $n=9$ factor 4.19, from $n=9$ to $n=10$ factor 4.04.) }
\end{conj}
}}

\section{On  enumerable classes of
 strand diagrams and plabic graphs}\label{ss:new-8}
%


%
%


\newcommand{\XX}{\mathcal X} 
\newcommand{\minimalist}{minimalist}
\newcommand{\rhombic}{rhombic} 

\newcommand{\Pv}[1]{\cite[#1]{Postnikov}}
\newcommand{\Pcor}[1]{\cite[Cor.14.2]{Postnikov}}

Let $\Pbr_n$ be the set of reduced plabic graphs 
\cite[\S 11]{Postnikov} 
of rank-$n$; 
and $\Po_n$ be the set of alternating strand diagrams 
as in~\cite[\S 14]{Postnikov}.
(See also Section~\ref{ssec:scottt-strand} and (\ref{de:plab1}).) 
Their relationship with $A_n$ can be summarized 
as follows: 

\[
\xymatrix{
& & A_n\ar@{^{(}->}[rd]^{\Scottt}\ar@{_{(}->}[ld]_{\Gst} \\
\XX_n\ar@{^{(}->}[r] & \Pbr_n\ar@<1ex>[rr]^D && \Po_n\ar@<1ex>[ll]^{D'} &\Pfim_n\ar@{_{(}->}[l]
}
\]
Here $\Gst$ is as in \S\ref{ss:intro}, $\Scottt$ as in \S\ref{ss:scott},
and $D,D'$ as in \S\ref{ssec:maps-str-plab}. 
In this section  
we apply Theorem~\ref{th:main} to 
corresponding subsets of 
plabic and strand diagrams. 
We define
 the sets $\XX_n$ of {\em \minimalist} strand diagrams, 
see Section~\ref{ssec:scottt-strand}; 
and $\Pfim_n$ of {\em \rhombic} (plabic) graphs, 
see (\ref{de:plab1}). 
We will show that these sets are in bijection with $A_n$.


For the sake of brevity we refer to Postnikov's original paper
for  motivations behind the
constructions of plabic and strand diagrams themselves.
These are large and complex classes of objects,
and canonical forms for them would be a useful tool.
%
%
The rigid/canonical nature of $A_n$   
induces
canonical forms for 
(the restricted cases of)
the other constructions. 


We start by characterizing the image of $\Scottt$ in 
Theorem~\ref{th:th3} as the 
set of \minimalist\ strand diagram and hence show that $\Scottt$ is injective. 
In Section~\ref{ssec:maps-str-plab} 
we recall Postnikov's
bijections between alternating strand diagrams and plabic graphs. 
(An illustration of the connection between plabic and strand diagrams is
given by Figure~\ref{fig:pleb1}(b).) 
This allows us to characterize the image of $\Gst$ in 
Section~\ref{ssec:tiling-plabic} 
as the set of \rhombic\ plabic graphs, Theorem~\ref{thm:iso-Pfim}.
Finally we determine the images of flip equivalence in the two other
realisations.

\begin{figure}
\raisebox{.431in}{\includegraphics[width=5.5cm]{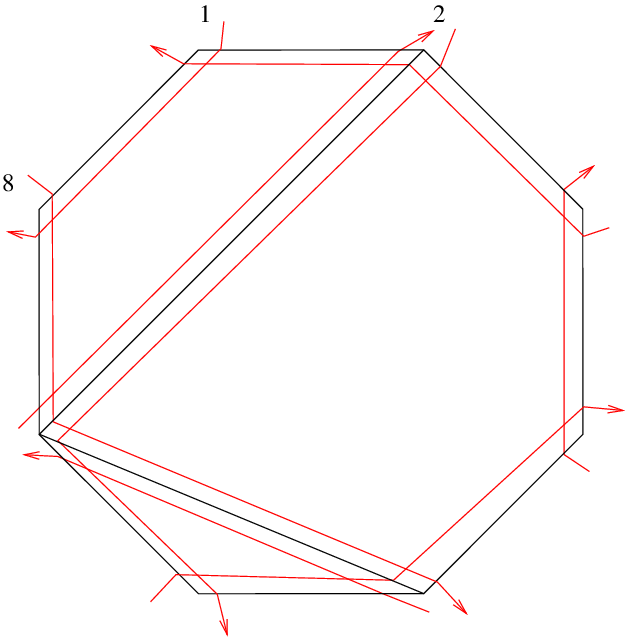}}
\raisebox{.431in}{\includegraphics[width=5.5cm]{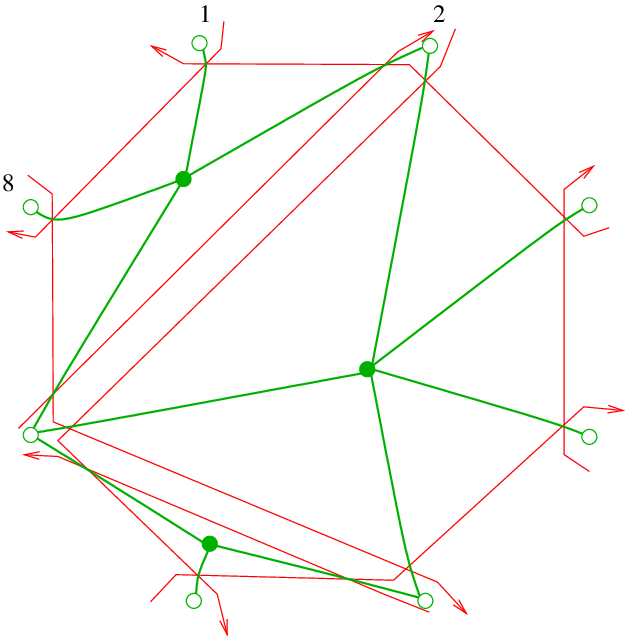}}
\caption{\label{fig:pleb1}
(a) A tiling $T$ (black) with strand diagram $\Scottt(T)$ of $T$; 
(b) the plabic graph ${D'}(\Scottt(T))$ (green). 
}
\end{figure}


\subsection{On $\Scottt$ and strand diagrams}\label{ssec:scottt-strand}
\ignore{{
Note that since the vertices of a Jordan diagram 
regarded as a graph are 4-valent 
the faces of the dual graph are quadrilateral, so it is 

Note that on the disk $D$ each strand $L_i$ of a Jordan diagram partitions 
$D \setminus L_i$ into two connected components. 
Label these components $\pm 1$.
Each point in $D \setminus \cup_i L_i$ can thus be labelled by the
product

Pick an arbitrary point in a face of a Jordan diagram on the disk.
Every other such point is 

It will be clear that every Jordan diagram on the disk can be given a
chequerboard colouring of its faces.

}}


\newcommand{\axiomi}{i}  
\newcommand{\axiomii}{i}
\newcommand{\axiomiii}{ii}
\newcommand{\axiomiv}{iii}
\newcommand{\axiomv}{iv}

An {\em absolute strand diagram} on $(S,M)$ is a 
\medial\ diagram such that: 
$\;$ 
(\axiomii)  strands 
   crossing a given strand 
   must alternate in direction;
(\axiomiii) if two strands 
   cross twice such as to cut out a simple disk
     then the resultant loop is oriented;
(\axiomiv) 
if a strand is self-crossing then no resultant loop is a simple disk;
(\axiomv) no strand is a closed loop cutting out a simple disk.


Note that this agrees with the ordinary definition of {\em \ASD} 
\cite{Postnikov,BaurKingMarsh} 
for $S$ a simple disk.
Here  rank $n= | M_\partial |=|M|$.

\ignore{{
\kub{don't we need to fix labels also on strand diagrams, i.e. something like the following in order 
to define our target set of strand diagrams 
and a map from them to $A_n$?}\\

\textcolor{magenta}{Note also, that in the polygon/disk with $n$ marked points, every \ASD\ 
has $n$ clockwise oriented regions on the boundary, with labels 
$1,2,\dots, n$ inherited by the labels of the \medial\ diagram, see Section~\ref{ss:scott}.} 

\medskip

\kub{lift the defn of $\Scottt(A_n)$ before the theorem as a new set, 
add the rank condition. (\textcolor{magenta}{done})\\
State proof of 8.3 as IFF, delete thm. write theorem saying the  
maps $\Scottt$ and new map are inverse to each other} 

\medskip
}}

\ignore{{
We remark that 
in general 
one could choose instead to equivalence under 
isotopies that fix boundary components setwise, but fix them pointwise
only at the ends of the motion (cf. \cite{BaurKingMarsh,LandoZvonkin04}).
It will be clear that this makes no difference to 
the disk case; or to $\Scottor(d)$ in general.
}}

\ignore{{
Let $\Po(S,M)$ denote the set of absolute strand diagrams on $(S,M)$ 
\kub{and $\Po_n$ for the set of absolute strand diagrams for the disked with $n$ marked points. }
}}
%
%


%
%





For any directed planar graph we classify the faces as 
clockwise, counterclockwise, 
alternating or other. 

\mdef
Let $\XX_n$ be the subset of rank-$n$ alternating strand diagrams 
whose faces are as follows: 
(i)
$n$ clockwise faces at the boundary, 
labelled $1,2,\dots, n$ going clockwise around the boundary; 
(ii) alternating faces 
with four sides; 
(iii) oriented faces in the interior that are counterclockwise and have 
at least 3 sides. 

We call the elements of $\XX_n$ {\em \minimalist\ strand diagrams}.
See Figure~\ref{fig:XX-example} for an example.  
\begin{figure}
\[
\includegraphics[width=5cm]{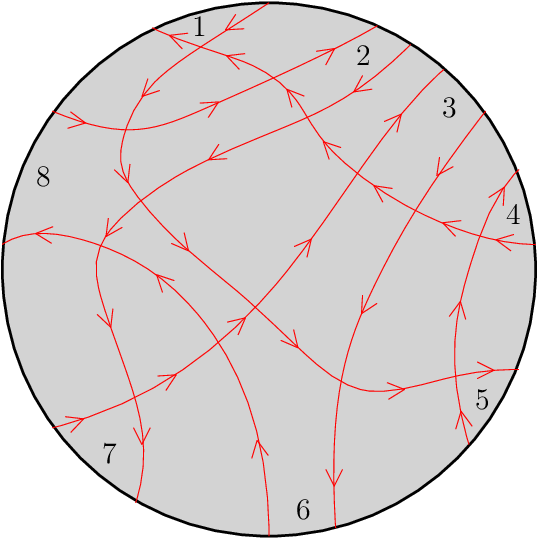}
\]
\caption{
A Jordan diagram on the 8-gon that is 
an element of $\XX_8$} \label{fig:XX-example}
\end{figure}


\mdef \label{pa:cheq}
Note that an element of $\XX_n$ (as every alternating strand diagram) 
has a chequerboard colouring of faces (see e.g. \cite{Postnikov}). 
If a clockwise face is black
(say) then all oriented faces are black and all alternating faces white.
Also the faces around an alternating face alternate
clockwise/counterclockwise. 
 

\newcommand{\ff}{{\mathsf f}}

\mdef
We define a `shrink' map 
$\ff:\XX_n\to A_n$ as follows: 
Let $d\in \XX_n$. 
Note from (\ref{pa:cheq})
that in $d$ regarded as an isotopy class of concrete diagrams
there are cases in which 
all the edges of clockwise faces are arbitrarily short.
Thus the clockwise faces are arbitrarily small neighbourhoods of $n$
points; and alternating edges have two short edges and two  edges
that pass between the clockwise faces (and hence are not short).
The paths of non-short edges are not constrained by the `shrinking' of
the clockwise edges.
Thus each pair may be brought close to each other, and hence 
form an arbitrarily narrow neighbourhood of a line
between two of the $n$ points.
Since no two alternating faces intersect, these lines cannot cross,
and so they form an element of $A_n$.

\medskip


\begin{theorem} \label{th:th3}
The map $\ff:\XX_n\to A_n$
is the inverse to a bijection  
$\Scottt:A_n\to \XX_n$. 
\end{theorem}

\begin{proof}
\ignore{{
That the image of $\Scottt$ is an alternating strand diagram 
follows from Lemma~\ref{lm:pf22}. 
It remains to see that  $\Scottt(A_n)\subseteq \XX_n$ and that 
compositions of the two maps are the identity. 
To see that $\Scottt$ maps tilings to elements of $\XX_n$ 
INSERT PROOF OF
\kub{old Lemma \ref{lm:properties-image}. FIX ME!} 

For the composition $\ff\circ \Scottt$: 
}}
It will be clear that $\ff$ makes sense on $\Scottt(T)$ since it even
makes sense tile by tile (cf. Fig.\ref{fig:p5inv1}). 
Indeed it recovers the tile,
so $\ff$ inverts $\Scottt$. 
The other steps have a similar flavour. 
\end{proof}

\ignore{{
Induction on the number of tiles. 
If $T$ is a single tile, the 
alternating regions in $\Scottt(T)$ are all boundary regions and
$f(\Scottt)$ is the untile, i.e.  
$\ff(\Scottt(T))=T$. 
Assume $T$ has $n\ge 2$ tiles. We know that $T$ has at least one ear
tile, say $t$ of size $r\ge 3$.  
Let the vertices of this ear be $i,i+1,\dots, i+r$ (reducing mod $n$). 
We have $\ff\circ\Scottt(t)=t$ 
(with vertices $i,i+1,\dots, i+r$). The tiling $T$ induces a tiling $T'$ on 
the polygon with the vertices 
$i+r,i+r+1,\dots, i$ (reducing mod $n$) and diagonal $[i,i+r]$. It has
$n-1$ tiles, so by induction  
hypothesis, 
$f\circ\Scottt(T')=T'$. 
The image $\Scottt(T)$ has an alternating 
region between the oriented regions of $i$ and of $i+r$. etc etc??? 
}}

\ignore{{
What about $\Scottt\circ \ff$? 

\end{proof}
}}

\smallskip

\noindent
{\bf Remark. } 
One can prove more generally, that $\Scottt$ is injective on tilings of 
$(S,M)$ and that the image of any tiling of $(S,M)$ is an absolute strand diagram. 
\ignore{
The injectivity follows as in the proof of Theorem~\ref{th:th3}. Let $T$ be a tiling of 
$(S,M)$. Then $\Scottt(T)$ is a \medial\ diagram (with transversal crossing) 
and it satisfies (\axiomii) by construction. (\axiomiii) holds as in the disk case. 
It remains to see that (\axiomiv) and (\axiomv) hold. 
Let $s$ be a strand in $\Scottt(T)$. If $s$ is self-crossing, $T$ contains a tile $t$ in which 
$s$ crosses near a vertex $p_0$. We need to see that $s$ cannot cut out a simple disk. 
If $s$ cuts out a loop containing $p_0$, then $p_0$ is a puncture in this loop 
and hence (\axiomiv) is satisfied. 
So consider the case where the loop is formed by the part of $s$ opposite $p_0$ and 
assume that this loop is a simple disk. In particular, it does not contain a boundary component. 
Then the loop is composed of at least three segments, where the loop crosses edges of the 
tiling. Since the loop cuts out a simple disk, the endpoints of these edges have to be outside 
this disk. But this is not possible (without creating monogons or digons). 
A similar argument shows that (\axiomv) holds. 
}


\pfig{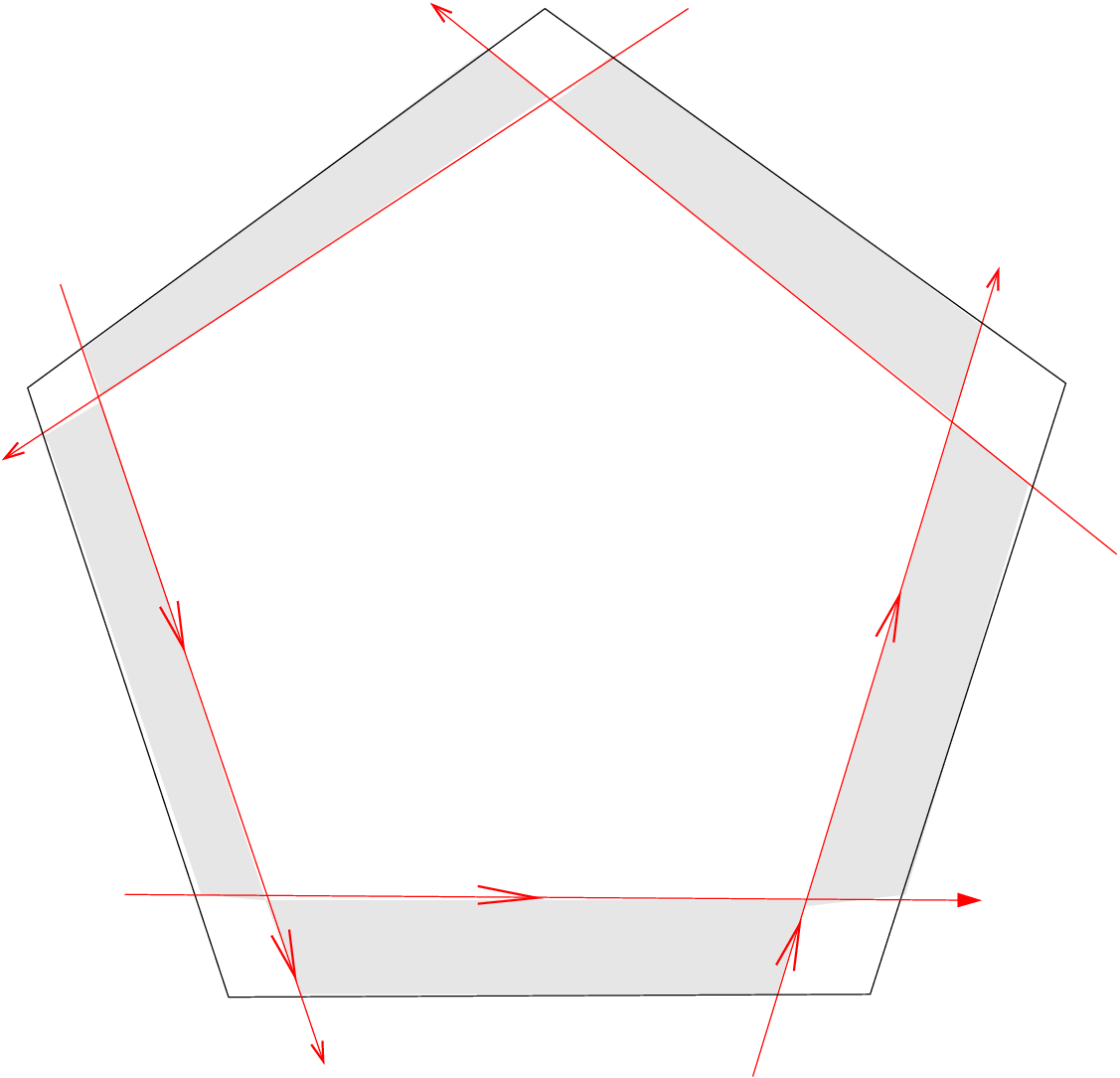}{Strands partition a tile
into vertex, edge and face parts.}

\ignore{{
\medskip

\kub{here some older bits}

\medskip

Noting Lemma~\ref{lm:pf22} 
we have
the following  
generalisation of \cite[Cor.2]{scott}.

\begin{theorem} \label{th:th3}
On $A_n$
the map $\Scottt$ is injective, 
and maps to alternating strand diagrams. 
\end{theorem}

\begin{proof}
That the image of $\Scottt$ is an alternating strand diagram 
follows from Lemma~\ref{lm:pf22}. It remains to prove the injectivity. 

Consider each $\Scottt(T)$ as a graph on $S$, which thus
divides $S$ up into faces.
Now consider for a moment keeping the tiling $T$
when drawing $\Scottt(T)$.
Then the strand segments
in a tile $t$ subdivide it into a neighbourhood of each
vertex (with clockwise oriented boundary),
a neighbourhood of each edge with alternating boundary orientation,
and a complementary region with counterclockwise boundary orientation
--- see Fig.\ref{fig:p5inv1}.
%
Observe that the alternating faces --- the
shaded faces in Fig.\ref{fig:p5inv1} ---
combine in  $\Scottt(T)$
to form alternating quadrilateral faces that `connect' vertex faces.
Of course these are present in  $\Scottt(T)$ without reference to $T$.
Thus  $\Scottt(T)$ determines a tiling.
We have constructed 
  an inverse to $\Scottt$. 
\end{proof}


\begin{lm}\label{lm:properties-image}
If an alternating strand diagram $d$ is in $ \Scottt(A_n)$
then:  
(i) The alternating \face s 
in the interior of $d$ have four sides; 
and the  `alternating' \face s on the boundary have three sides. 
(ii) The oriented \face s in the interior are counterclockwise and have at least $3$ sides; 
on the boundary they are clockwise.  
\end{lm}
\begin{proof}
Consider a tiling $T$ such that $d=\Scottt(T)$. 
Let $t$ be a tile of size $r$ in $T$, then 
the strands of 
$\Scottt(t)$ divide $t$ up into alcoves. Clearly one of these is a
complete 
counterclockwise \face\  of $d$
formed by $r$ strand segments, hence with $r$ sides. 
There is also an alcove at each vertex.
The two strand segments in its boundary are oriented clockwise.
There is also an alcove for each edge.
The three strand segments in its boundary have alternating orientation.
It will be clear 
that the disjoint union of these three types of alcove is all the
alcoves of $\Scottt(t)$.


Now we look at all $t$ in $T$ together. The tiles meet at the diagonals 
of $T$. Whenever two tiles share an edge to form a diagonal of $T$, 
the respective alcoves at this common edge form 
a four-sided alternating face. 
At the vertices of $T$, the alcoves of the tiles form a clockwise region with 
the number of sides equal to 1 + the number of tiles incident with the vertex. 

This proves the lemma. 
%
\end{proof}

}}

\newcommand{\frpg}{fully reduced plabic graph}
\newcommand{\Frpg}{Fully reduced plabic graph}


\subsection{Maps $D,D'$ between strand diagrams and plabic graphs}
\label{ssec:maps-str-plab} $\;$




\mdef \label{de:plab1}
A {\em plabic graph} $\pgr$ is 
 a planar, disk-embedded undirected graph with two `colours' of
vertices/nodes, considered up to homotopy \Pv{Definition 11.5}. 
Vertices are allowed on the disk boundary. 
The rank of $\gamma$ is the number of these `tagged' vertices.
In rank $n$ they are labelled $\{1,2,...,n \}$ clockwise.
%


Postnikov defines 
`moves' on plabic graphs in \Pv{\S12}: 

\[
\includegraphics[scale=.45]{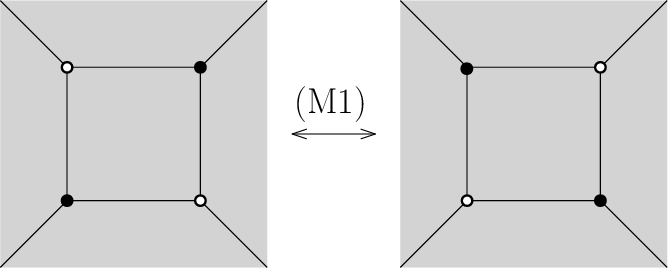} 
\hskip .5cm
\includegraphics[scale=.45]{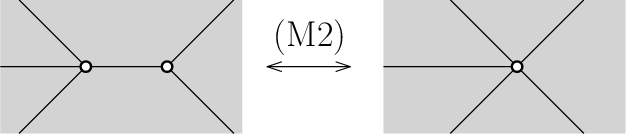} 
\hskip .5cm
\includegraphics[scale=.45]{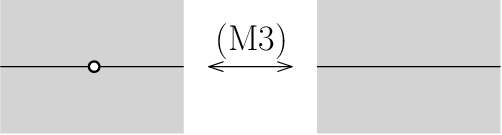} 
\]
with M2-3 also for black nodes.
In M2 any number of incoming edges is allowed.
Postnikov also defines reductions on plabic graphs:
\[
\includegraphics[scale=.6]{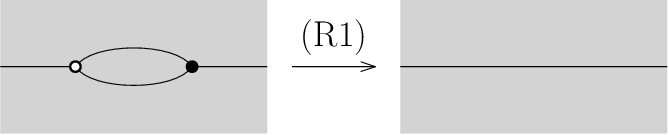} 
\hskip 1.5cm
\includegraphics[scale=.6]{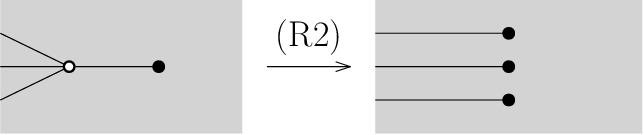} 
\]
and similarly with colours reversed.
The {\em move-equivalence class} of $\pgr$ is
its orbit under (M1-3). 
A plabic graph of rank $n$ is {\em reduced} if it 
has no connected component without boundary vertices;
and if there is no graph in its move-equivalence class 
to which (R1) or (R2) can be applied. 
See \Pv{\S12} for details.

We write $\Pbr_n$ for the set of reduced plabic graphs 
of rank $n$.



\smallskip

Recall the map $\Gst$ on $A_n$ to plabic graphs from \S\ref{ss:intro}.
If $T$ is a tiling of an $n$-gon, we draw a white node at each vertex of
the polygon and a black node in each tile, connecting the latter 
by edges with the white nodes at the  vertices of the tile. 
One can see that the graph produced has no parallel bicoloured edges
and no internal leaves with bicoloured edges.
Thus $\Gst:A_n\to \Pbr_n$.

Postnikov's plabic {\em networks}
are generalisations of the above including face
weights.
Here it will be convenient to consider another kind of generalisation. 




\mdef
For any planar graph $L$ there is a {\em medial graph} $\mm(L)$
(see e.g. \cite[\S12.3]{Baxter82}), 
which is a planar graph distinct from but overlaying $L$. 
We obtain $\mm(L)$ by drawing a vertex $\mm(e)$ on each edge $e$ of $L$,
then whenever edges $e,e'$ of $L$ are incident at $v$ and bound
the same face we draw an edge $\mm(e)$-$\mm(e')$.

\mdef \label{de:DL}
For $\mm(L) $  we note the following. 
(1) $\mm(L)$ has 
a polygonal face $p_v$ around each vertex $v$ of $L$.
(2) Monogon and digon faces are allowed --- see Figure~\ref{fig:D1}
(so edges may not be straight).
(3)
The faces of $\mm(L)$ are of two types: 
containing a vertex of $L$, 
or not. 
Given an asignment of a colour (black/white) to each vertex of $L$ 
then we 
get a digraph  $\overrightarrow{\mm}(L)$ by asigning an orientation to each 
polynomial face:
counterclockwise if $v$ is black and clockwise otherwise. 
(4)
If $L$ is bipartite and indeed 2-coloured then 
for this asignment
the orientations in $\overrightarrow{\mm}(L)$ have
the property that we may reinterpret the collection of meeting
oriented polygons 
as a collection of crossing oriented strands, denoted $D_L$.

\mdef \label{de:mvtag} 
Suppose $L$ has some labelled exterior vertices.
A `half-edge' or `tag' may be attached to any such vertex $v$ 
(specifically one usually thinks of $L$ bounded in a disk in the
plane, and the tag as an edge passing out through the boundary)
whereupon there is a medial vertex $\mm(v)$ on the half-edge,
and the (exterior) medial
edge around $v$ becomes two segments incident at $\mm(v)$.
In this case, if $v$ is labelled in $L$ then we say that $\mm(v)$
inherits this label in $\mm(L)$. 

\mdef
Noting (\ref{de:DL}) and (\ref{de:mvtag}), the map 
$$
D:\Pbr_n \rightarrow \Po_n
$$ 
may be defined by $D(L)=D_L$.


\newcommand{\ing}[2]{\includegraphics[width= #1 cm]{\newfigd #2.eps}}
\begin{figure}
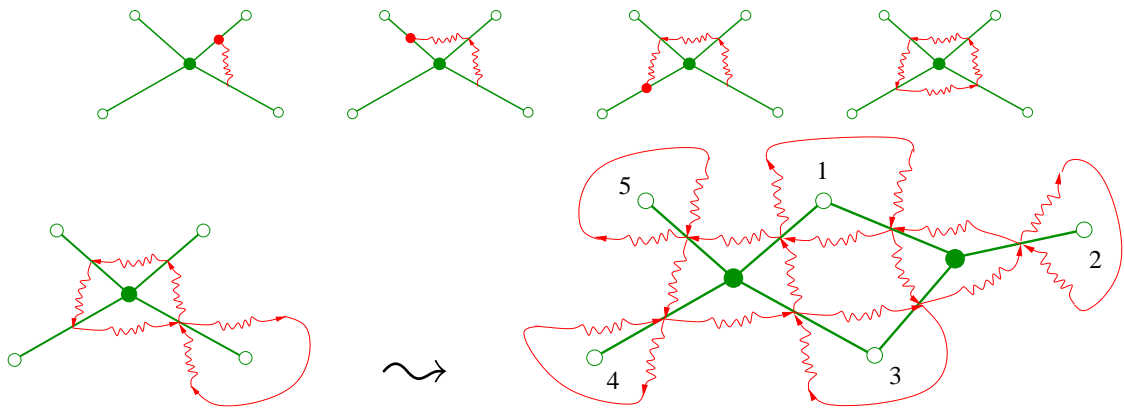

\[
\ing{2.46}{Dmap2} \qquad
\ing{2.46}{Dmap3} \qquad
\ing{2.46}{Dmap4} \qquad
\ing{2.46}{Dmap5}
\] 
\[
\ing{4.0}{Dmap6} \hspace{.32in}
\raisebox{.1in}{ {\huge $\leadsto$} } \hspace{.32in}
\ing{7.9}{Dmap8}
\]
\caption{\label{fig:D1} Constructing the $D$-map
}
\end{figure}


\mdef A {\em fully reduced} plabic graph is a reduced plabic graph
without non-boundary leaves; and without unicolored edges.
In particular it is a connected 2-coloured planar graph.
Write $\Pf_n$ for the set of fully reduced plabic graphs of rank $n$. 



Postnikov's Corollary 14.2(1) 
can now be summarized as: $L \mapsto D_L$ 
restricts to a bijection  $D: \Pf_n \rightarrow \Po_n$.


\mdef
Postnikov gives a map 
$$
D': \Po_n \rightarrow \Pf_n
$$
as follows, that inverts $D$.
\label{def:D'-map}
Let $d$ be an alternating strand diagram. Then $D'(d)=\pgr_d$ is the plabic 
graph we obtain by drawing 
a white vertex in each clockwise oriented \face\ and 
a black vertex in each counterclockwise \face. 
Two vertices are connected by an edge if and only if their faces
are opposite each other at the crossing point of a 
pair of crossing strands. 
(Example:
Figure~\ref{fig:pleb1}.) 



\subsection{Properties of the map $\Gst$}  
\label{ssec:tiling-plabic}

$\;$ 

We note that $\Gst$ is the composition $D' \circ \Scottt$. Since $D'$ is a bijection and 
$\Scottt$ is injective (Theorem~\ref{th:th3}), $\Gst$ is injective. 
In this section, we give an intrinsic characterization of the image of $\Gst$.

\mdef 
Let $u$ and $v$ be two black nodes in $\pgr\in \Pf_n$ that are on a common quadrilateral. 
If $u$ has degree $r+2$ and is incident with $r\ge 1$ leaves, 
we say that $\pgr$ has an {\em $r$-\bouquet} 
at $u$ or a {\em \bouquet} at $u$. The subgraph on the quadrilateral 
and on the $r$ leaves is the \bouquet\ at $u$. 

The first figure below is a bouquet at $u$ with 
4 leaves. The second figure shows two (non-disjoint)
bouquets, one at $u$ and one at $v$. The second graph  has two \bouquet s. It 
satisfies the conditions for $\Pfim_n$ of Definition~\ref{def:G-image}.

\[
\includegraphics[scale=.5]{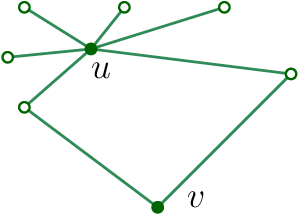}  
\hskip 1cm 
\includegraphics[scale=.5]{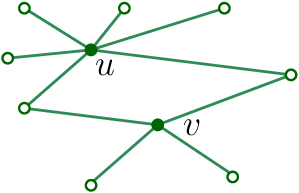} 
\]



\newcommand{\axioma}{a} 
\newcommand{\axiomb}{b} 
\newcommand{\axiomc}{c} 
\newcommand{\axiomd}{e} 
\newcommand{\axiome}{d} 
\newcommand{\node}{node}

\begin{defn}\label{def:G-image}
The set 
$\Pfim_n$ of {\em \rhombic\ graphs} is 

the set of connected fully reduced plabic graphs $\pgr$ in $\Pf_n$
containing at least one black \node\ and such that  

\noindent
(\axioma) the tagged \node s (in the sense of (\ref{de:plab1}))
are white and all other \node s are black,  
\\
(\axiomb) every black \node\ has degree $\ge 3$, 
\\
(\axiomc) 
every closed face is a quadrilateral,

\noindent
(\axiome) 
in the fan of edges coming out of a white node every adjacent pair is part of a quadrilateral. 
\end{defn}


\mdef \label{eq:axiomd}
We observe that conditions (\axioma) and (\axiomb) imply: 
%
(\axiomd)
Two faces of a \rhombic\ graph 
share at most one edge.

\smallskip

For $n=3,4,5$, $\Pfim_n$ has 1,3,11 elements respectively,
cf. Figure~\ref{fig:pfim-low-rank}. 
\begin{figure}
\[
\includegraphics[scale=.52]{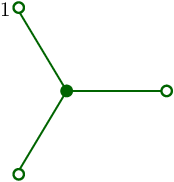}
\hskip 1cm  
\includegraphics[scale=.65]{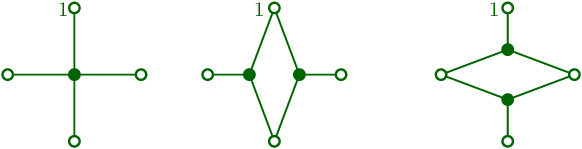}
\]
\[
\includegraphics[scale=.65]{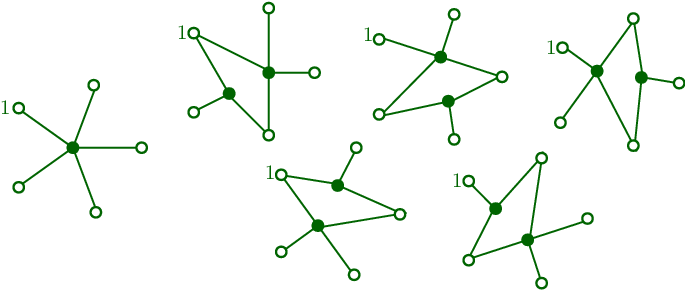}
\hskip .2cm  
\includegraphics[scale=.65]{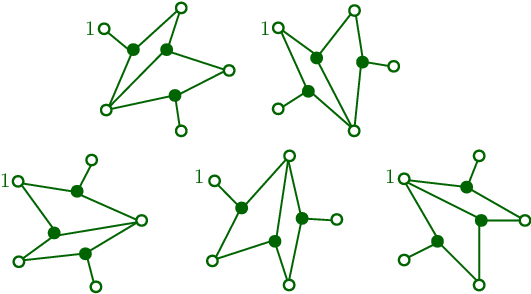}
\]
\caption{$\Pfim_n$ in ranks $n=3,4,5$}\label{fig:pfim-low-rank}
\end{figure}



\begin{lm}
\label{lm:yay}  
If $\pgr\in \Pfim_n$, $\pgr$ not a star, 
then $\pgr$ has at least two \bouquet s.
\end{lm}
\begin{proof}
Forget the leaves for a moment, so we have graph of quadrilaterals. 
Now consider the exterior `face` subgraph - a 2-coloured loop.
We see (e.g. by induction on number of faces, using
(\ref{eq:axiomd})) 
that 

this must have at least 2 black corners 
(black nodes touching only 1 quadrilateral).
\end{proof}


We note that $\Gst(A_n) \subseteq \Pfim_n$.
Our next goal is to get an inverse to the map $\Gst$, going from \rhombic\ graphs 
to tilings.  
One ingredient is the following lemma which says that if we split an
element of $\Pfim_n$  
at a \bouquet\ at node $u$, we obtain a star graph and an element $\pgr_u$ 
of $\Pfim_n$. 

\mdef
Let $\pgr$ be a plabic graph 
containing a \bouquet\ at vertex $u$, with $u$ of degree $r+2$. 
Define $\pgr_u$ as the full subgraph on the vertex set excluding
$u$ and its leaves. 
We denote by $\pgr_s$ the full subgraph on $u$ and all white nodes 
incident with $u$. 

For example here $\pgr_s$ is the upper graph on the right and 
$\pgr_u$ is the lower graph on the right. 

\[
\includegraphics[scale=.6]{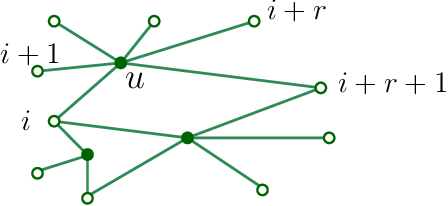}  
\hskip 1cm  \stackrel{\mbox{\tiny split}}{\leadsto} \hskip 1cm
\includegraphics[scale=.6]{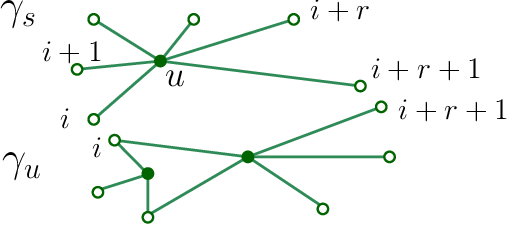} 
\]



\begin{lm} \label{lm:bouquet-removal}
Let $\pgr\in \Pfim_n$, $\pgr$  not a star. 
If $\pgr$ has a bouquet at 
$u$   
then $\pgr_u\in\Pfim_n$. 
\end{lm}

\begin{proof}
Note that $\pgr_u$ inherits  
(\axioma) and (\axiomb) of Definition~\ref{def:G-image} from $\pgr$. 
Denote the second black 
node of the \bouquet\ at $u$ by $v$. 
When splitting, the quadrilateral face involving $u$ and $v$ becomes 
a boundary face of $\pgr_u$. 
All other faces of $\pgr_u$ 
are faces of $\pgr$. So (\axiomc) also holds. 

It remains to see that (\axiome) holds for $\pgr_u$. The 
only vertices to check are $i+1$ and $i+r+1$. 
In $\pgr$, every adjacent pair of edges at $i+1$ (or at $i+r+1$ respectively) 
are part of a quadrilateral. When going to $\pgr_u$, one extremal edge of the 
fan is removed, so the remaining edges still satisfy (\axiome). 
\end{proof}

\mdef 
\label{def-Gprime}
Let $\pgr\in\Pfim_n$.  
Consider the set 
\[
\{[i,j]\mid \mbox{$i$ and $j$ are white nodes in a quadrilateral of $\pgr$}\} 
\]
Note
(by \ref{def:G-image}(\axioma) and the construction)
that 
this forms a collection of pairwise non-crossing diagonals of an $n$-gon. 
We denote this tiling $T=\Gprime(\pgr)$. 


\begin{theorem}\label{thm:iso-Pfim}
The map $\Gprime$ is the inverse to a bijection $\Gst$: $A_n\to \Pfim_n$. 
\end{theorem}

\begin{proof}
We will show that $\Gst\Gprime(\pgr)=\pgr$ for every $\pgr\in\Pfim_n$. 


We use induction on $n$. 
If $|\pgr|=3$, by Definition~\ref{def:G-image}, 
$\pgr$ does not contain any quadrilaterals, hence is a star, and 
$\Gprime(\pgr)$ is the untile $T$ of a triangle, with 
$G(T)=\pgr$. 
So assume that the claim is true for $\Pfim_{n-1}$. 
Take $\pgr\in\Pfim_n$. 
%
If $\pgr$ is a star, $T=\Gprime(\pgr)$ is the untile 
of the $n$-gon and $\Gst(T)=\pgr$. 
So assume $\pgr$ is not a star. By Lemma \ref{lm:yay}, it then contains at least 
two \bouquet s, say an 
$r$-\bouquet\ for some $1\le r< n-2$.

We split $\pgr$ at the 
\bouquet\ and obtain a star $\pgr_s$ and the graph $\pgr_u$. 
Let the white nodes of this 
star be $i,i+1,\dots, i+r+1$ (reducing mod $n$). 
Then the white nodes of $\pgr_u$ are 
$i+r+1, i+r+2,\dots, i$ (reducing mod $n$). 
Graphs $\pgr_s$ and $\pgr_u$ are elements of 
$\Pfim_{r+2}$ (with $r+2<n$) and $\Pfim_{n-r}$ respectively by 
Lemma~\ref{lm:bouquet-removal}. 
So by induction for the tilings $T_s=\Gprime(\pgr_s)$ and 
$T_u=\Gprime(\pgr_u)$ of polygons we have $\Gst(T_s)=\pgr_s$ and 
$\Gst(T_u)=\pgr_u$. 

Tiling
$T_s$ is the untile of the polygon $P_s$ on the vertices $i,i+1,\dots, i+r+1$; 
$T_u=\Gprime(\pgr_u)$ a tiling of 
the polygon $P_u$ on the vertices $i+r+1, i+r+2,\dots, i$. 

We glue the two polygons $P_s$ and $P_u$ along the boundary edges 
$[i+r+1,i]$ and $[i,i+r+1]$ to obtain an $n$-gon $P$ with vertices $1,2,\dots, n$  
and tiling $T$ given by 
the union of the diagonals of $T_s$, $T_u$ and diagonal $[i,i+r+1]$. 

Since $T$ contains a diagonal exactly for every quadrilateral in 
$\pgr$, $T=\Gprime(\pgr)$. 
By construction, $\Gst(T)=\pgr$. 
\end{proof}

\subsection{$\Pfim_n$ and equivalence classes under moves}

$\ $

\newcommand{\rhod}{\rho_{\diamond}} 

\mdef
We define  `moves' $\rhod$ 
on elements of $\Pfim_n$ as 
in Fig.\ref{fig:move-Pt}
(these moves are a particular combination of M1 and M2 from 
\S\ref{ssec:maps-str-plab}). 

\begin{figure}
\[
\includegraphics[scale=.5]{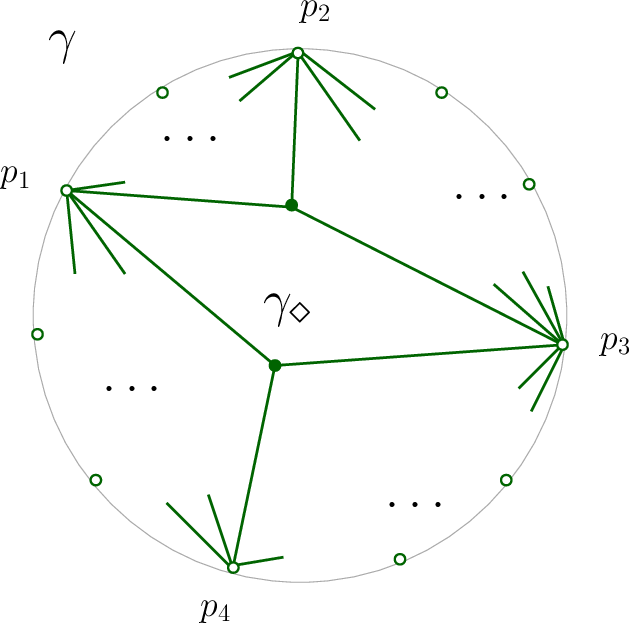} 
\hskip .5cm \stackrel{\rho_{\diamond}}{\longleftrightarrow}
\includegraphics[scale=.5]{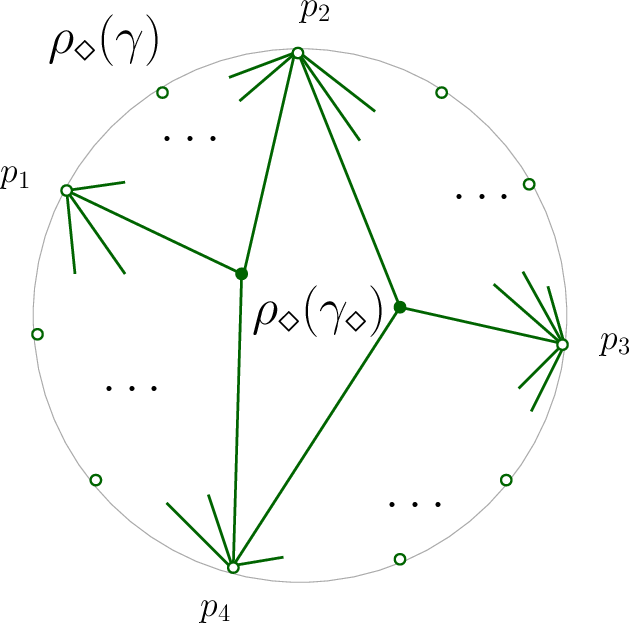}
\]
\caption{Move $\rhod$ in $\Pfim_n$}\label{fig:move-Pt}
\end{figure}


\begin{figure}
\[
\includegraphics[scale=.52]{\newfigd Pt3.eps}
\hskip .4cm  
\includegraphics[scale=.65]{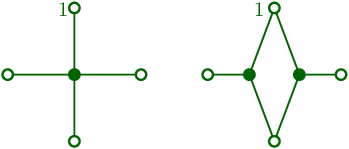}
\]
\[
\includegraphics[scale=.65]{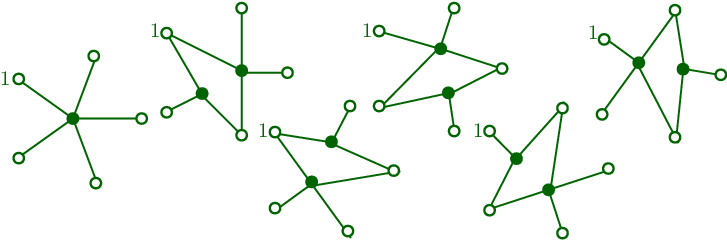}
\hskip .4cm  
\includegraphics[scale=.65]{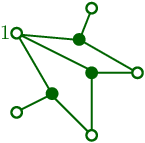}
\]
\caption{Classes $\Pfim_n/\sim$ in ranks $n=3,4,5$}\label{fig:pfim-equiv}
\end{figure}


\begin{lm}\label{lm:move-flip}
Under the bijection $\Gprime:\Pfim_n\to A_n$, a move $\rhod$ corresponds 
to a flip in a tiling. 
\end{lm}
\begin{proof}
Consider $\pgr\in\Pfim_n$, let $T=\Gprime(\pgr)$. 
Any quadrilateral $\diamond$ in $\pgr$ corresponds to a diagonal
$[p_1,p_3]$ in $T$. 
Now assume 
that two black nodes $u_1,u_2$ of the quadrilateral $\diamond$ have degree three and let $p_2$, $p_4$ 
be the other two white nodes adjacent to the black nodes of
$\diamond$. 
Then the four full 
subgraphs $p_i,u_j,p_{i+1}$ (with $i,j$ appropriate) 
of $\pgr$ are either boundary paths or part of quadrilaterals. 
In the former case, 
the image of $\pgr$ under $\Gprime$ has a boundary segment
$[p_i,p_{i+1}]$; 
in the latter case, it has a diagonal $[p_i,p_{i+1}]$. 
In any case, $T$ contains a triangulated quadrilateral $p_1,p_2,p_3,p_4$ with diagonal $[p_1,p_3]$ 
and the move $\rho_{\diamond}$ corresponds to the exchange $[p_1,p_3]\longleftrightarrow [p_2,p_4]$ 
in $T$. 
\end{proof}

Given Lemma~\ref{lm:move-flip}
we can then define move-$\rhod$ equivalence classes 
 $\Pfim_n/\!\!\!\sim$
on $\Pfim_n$. 
Furthermore, 
the number of equivalence classes are the same as $|\AE_n|$. 

For $n=3,4,5$, one can readily confirm $1,2,7$ classes respectively 
using Figure~\ref{fig:pfim-equiv}. (Although this does not
provide any obvious new method to compute in higher ranks, 
cf. \S\ref{sec:count}.)


\mdef
On strand diagrams, flip corresponds to a combination of moves 
from Figure~\ref{fig:sdmoves1} 
(recalled from~\cite[\S14]{Postnikov})  
--- the combination given in  
\S\ref{ss:scott}. 

\begin{figure}
\[
\includegraphics[width=4cm]{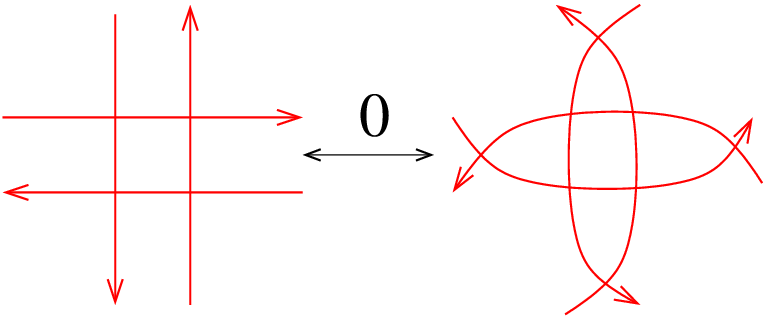}
\hskip 2.1cm
\includegraphics[width=4cm]{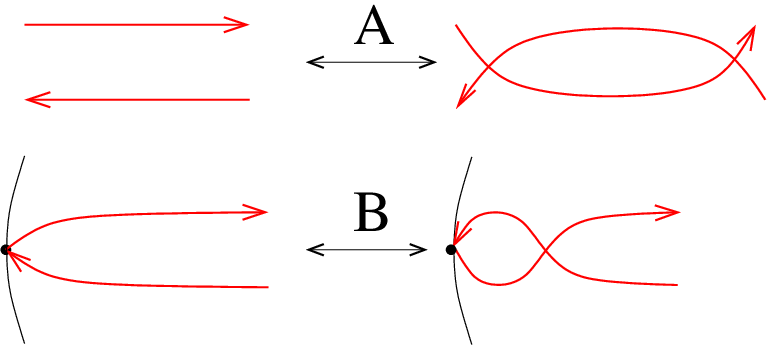}
\]
\caption{\label{fig:sdmoves1}
Postnikov's moves of alternating strand diagrams}
\end{figure}



\mdef
There are many beautiful set sequences in the little Schr\"oder
combinatoric \cite{Stanley97}
($A_n$ is a standard one, to which we have now added 
$\Pfim_n$ and $\XX_n$).
It is one nice problem 
for future consideration to recast
flip equivalence into cases such as Schr\"oder's original bracket
sequences  
(the set $\Omega_n$ is the set of properly nested
bracketings on a word of length $n-1$, where there is always an
(undrawn) outer bracketing and otherwise each bracket pair must
contain at least two symbols
--- the bijection with $A_n$ is elementary via
rooted versions of 
the dual trees of
\S\ref{ss:tree1}).
As a taste of this game, the first few in this case are as
follows:
$
\Omega_3 = \{ ab \}, \; 
\Omega_4 = \{ abc,(ab)c,a(bc) \}, \; 
$
and in A-complex form


\[
\Omega_5  \; = \;\;\;\;  \;\;
\xymatrix{
  ((ab)c)d \ar[d] \ar[drrr] & a((bc)d) \ar[d] \ar[drrr]
  & (ab)(cd) \ar[d] \ar[dll]
  & (a(bc))d \ar[d] \ar[dll] & a(b(cd)) \ar[d] \ar[dll]
  \\
  (ab)cd \ar[drr]        & a(bc)d \ar[dr] & ab(cd) \ar[d] & (abc)d
  \ar[dl] & a(bcd) \ar[dll]
  \\
                        &               & abcd
  }
\]


\ignore{{
\[
\Omega_5  \; = \ \ 
\xymatrix@C=-1em{
  a((bc)d) \ar[dr] \ar@{~}[rr]  && a(b(cd)) \ar[dr] \ar[dl]\ar@{~}[rr]
  && (ab)(cd) \ar[dr] \ar[dl]\ar@{~}[rr] && ((ab)c)d \ar[dr] \ar[dl]\ar@{~}[rr]
  && (a(bc))d \ar[dr] \ar[dl]\ar@{~}[rr] && \textcolor{gray}{a((bc)d)} \ar[dl] 
  \\
        &a(bcd) \ar[drrrr] && ab(cd) \ar[drr] && (ab)cd
  \ar[d] && (abc)d \ar[dll] && a(bc)d \ar[dllll] 
  \\
                        &&&& &   abcd
  }
\]
}}




\noindent Here flip equivalence collapses the entire first row to a point.








\medskip\medskip

\noindent {\bf Acknowledgements}.
We thank Henning H. Andersen, Aslak B. Buan and Walter Mazorchuk and the Mittag-Leffler
Institute for inviting us to the 2015 Semester on Representation Theory,
where this work was started.
We thank  Robert Marsh and Hannah Vogel for useful comments
on the manuscript. 
PM would like to thank EPSRC for partial funding under the grant
EP/I038683/1
``Algebraic, Geometric and Physical underpinnings of Topological
Quantum Computation''. 
KB would like to thank FWF for partial funding under projects 
P25141-N26, P 25647-N26 and W1230.

\bibliographystyle{amsplain}
\bibliography{bm}


\newpage


\appendix 

\section{The number of `Scott permutations'}

\noindent
\begin{center}
\begin{footnotesize}
{\sc Max Glick$^{\dagger}$ (as interpreted by KB, PPM)} 
\end{footnotesize}\end{center}


\newcommand{\beq}{\begin{equation}}
\newcommand{\eq}{\end{equation}}





Here
we will determine a generating function for the number of polygon
tilings up to flip equivalence and hence compute the asymptotic growth rate. 
As a warm-up we first recall the case for {\em all} tilings
--- the little Schr\"oder numbers. 

By $X_n$ we denote the set of  tilings of an $(n+2)$-gon, a (convex) polygon with
$n+2$ vertices. 
By $x_n$ we denote the number $|X_n|$.
Define
$
X(z)=\sum_{n=0}^{\infty} x_n z^n ,
$
with $x_0=1$. \\

Fix a `base' edge of the $(n+2)$-gon. Then we can decompose $X_n$
according to the number $r$ of edges in the tile
that meets this edge.
Writing $X_{n,r}$ for the subsets, $x_{n,r}$ for their sizes,
and $X^{(r)}$ for the generating functions, we have
\beq \label{eq:ArzAr}
X^{(r)} = z^{r-2} X^{r-1}     \hspace{1in} (r\geq 3)
\eq
since we may construct a tiling by attaching a tiling 
(by its base)
to each of the
non-base edges of the base tile. \\

Thus we obtain the standard results:
\beq
X= \sum_{r\ge 2}X^{(r)}=1+zX^2 + z^2X^3 + z^3X^4 + \dots = 1 + \frac{zX^2}{1-zX}
\eq
and hence
\[
X(z)=\sum_{n\ge 0}x_nz^n= \frac{z+1 - \sqrt{z^2-6z+1}}{4z}
\]

\medskip

\newcommand{\AED}{D}  

Now we turn to $\AE_n$. Write $d_n = |\AE_{n+2}|$
and $\AED(z) = \sum_n d_n z^n$.
We continue to hold fixed a base edge of the $(n+2)$-gon.
Note that an element of $\AE_{n+2}$ is now a {\em class} of tilings,
but
the number $r$ of edges of the tile incident to the base continues to
be well-defined. Thus we can partition $\AE_{n+2}$ into subsets
$\AE_{n+2,r}$ according to $r$.
Write $d_{n,r}$ and $\AED^{(r)}$ as above. 
%
We have
\[
\AED^{(r)}  = z^{r-2} \AED^{r-1} \hspace{1in} (r \geq 4 )
\]
by an analogous argument to the $A$-case
in (\ref{eq:ArzAr}).
However, the case $r=3$ is made more complicated by the equivalence relation.  For convenience let $b_n=d_{n,3}$ and $c_n = d_n-b_n$.  The corresponding generating functions are $B = \AED^{(3)}$ and 
\beq \label{eq:C1D}
C = 1+ \sum_{r \geq 4} \AED^{(r)} = 1 + z^2D^3 + z^3D^4 + \dots
\eq
so that
\beq
\AED 
 = B + C
\eq
We may then write
\beq \label{eq:BzDC}
B = z \AED C
\eq
Proof of (\ref{eq:BzDC}): 
Consider a class of tilings whose base tile is a triangle and a representative therein.
Here the base triangle lies in some connected triangulated part.
Note that we can always choose a representative so that this triangle
is the rightmost triangle in this connected triangulated part.
(This choice breaks an overall symmetry in the set, but this need not concern
us.)
The result now follows
on noting that we may attach a (representative) tiling
with any base on the
left, and one from the subset with non-triangle base on the right.
\qed



To illustrate (\ref{eq:BzDC}) schematically: 
take a triangle with a distinguished base edge, glue tilings to
the left hand edge and tilings with a tile of size $\ge 4$ to the
right hand  edge. 
\[
\includegraphics[width=6cm]{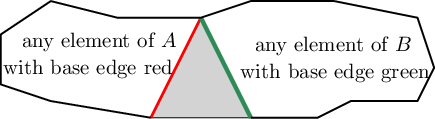}
\]
Example:
This gives 
$b_4 = d_3c_0 + d_2c_1 + d_1c_2 + d_0c_3$ (note that
$c_1=0$), or, in tilings, with the base triangle shaded:  
\[
\includegraphics[scale=.5]{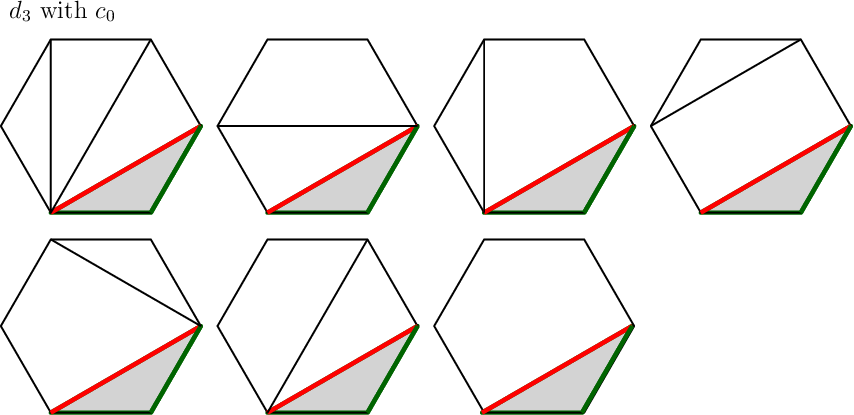}
\]
\[
\includegraphics[scale=.5]{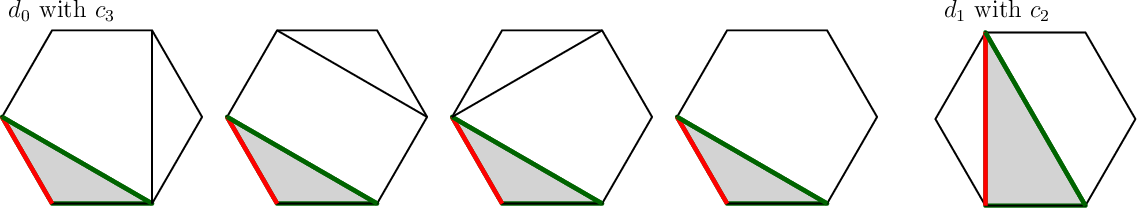}
\]


Eliminating $B,C$ from our formulae we obtain 
a quartic for $D$:
\beq \label{eq:quartic}
  Y^4 + Y^3 + Y^2(1-x) -Y + x = 0
\eq
where $Y=xD$.
Computing the discriminant yields a dominant singularity of $D$
at 0.19448....
Thus the asymptotic ratio of coefficients is the reciprocal
\[
\lim_{n\rightarrow\infty} \frac{d_{n+1}}{d_n} = 5.1418....
\]
(Finer details of the asymptotic behaviour can be determined --- see
for example \cite{FS}.)
One can compare this with the asymptotic ratio for the little Schr\"oder
numbers which is famously $3+2\sqrt{2} = 5.8284... $;
and of course to the asymptotic ratio of sizes of the sets of all
permutations, which is unbounded.

Equation \eqref{eq:quartic} gives rise to a recurrence for the 
$d_n$ making it possible to compute several terms easily. Here, we list the first 15 terms: 
\[
\begin{array}{|r|rrrrrrrrrrr|}
\hline
n   & 1 & 2 & 3 & 4 & 5 & 6 & 7 & 8 & 9 & 10  & \phantom{}
\\
\hline
d_n & 1 & 2 & 7 & 26 & 100 & 404 & 7254 & 31726 & 140964 & 634506 
     & \\
\hline
\end{array}
\]
%
\[
\begin{array}{|r|rrrrr|}
\hline
n   & 11 & 12 & 13 & 14 & 15  
\\
\hline
d_n 
     & 2887168 & 13258914 &  61373864 &  286053987 & 1341325126  \\
\hline
\end{array}
\]

\medskip

\phantom{bla}

\hfill

\noindent Acknowledgement. KB thanks Peter Grabner for useful conversations.


\noindent
\begin{footnotesize}
$^{\dagger}$ {\sc Department of Mathematics, University of Connecticut, Storrs, Connecticut, USA}
\end{footnotesize} 

\medskip


\end{document}